%% file: main.tex
	\documentclass[aop,preprint]{imsart}
    \setattribute{journal}{name}{}

    \usepackage[utf8]{inputenc}
    \usepackage{amsmath,amsthm,amsfonts,hyperref}
    \usepackage{multirow,rotating}
	\usepackage{tikz,tkz-euclide,pdftexcmds}
    \usepackage[abs]{overpic}

    \newtheorem{theorem}{Theorem}
    \newtheorem{corollary}[theorem]{Corollary}
    
    \newtheorem{remark}[theorem]{Remark}
    \newtheorem{lemma}[theorem]{Lemma}
    \newtheorem{conjecture}[theorem]{Conjecture}

	\tikzset{>=latex}
	\usetikzlibrary{backgrounds}
	\pgfdeclarelayer{foreground}
	\pgfsetlayers{background,main,foreground}
	\usetkzobj{all}
    \makeatletter
		\pgfdeclareradialshading[tikz@ball]{my ball}{\pgfqpoint{5bp}{10bp}}{%
			color(0bp)=(tikz@ball!10!white);
			color(8bp)=(tikz@ball!65!white);
			color(12bp)=(tikz@ball);
			color(25bp)=(tikz@ball!70!black);
			color(50bp)=(black)%
		}
		\pgfdeclareradialshading[tikz@ball]{new ball}{\pgfqpoint{12bp}{12bp}}{%
			color(0cm)=(tikz@ball!10!white);
			color(0.5cm)=(tikz@ball!50!white);
			color(0.75cm)=(tikz@ball);
			color(0.9cm)=(tikz@ball!80!black);
			color(1cm)=(tikz@ball!60!black)%
		}
		\pgfdeclareradialshading{urn shadow}{\pgfpointorigin}{%
        	color(0cm)=(black);
			color(2mm)=(gray!70);
			color(3mm)=(gray!30);
			color(7mm)=(white)%
		}

		\tikzoption{my ball color}{\pgfutil@colorlet{tikz@ball}{#1}\def\tikz@shading{my ball}\tikz@addmode{\tikz@mode@shadetrue}}
		\tikzoption{new ball color}{\pgfutil@colorlet{tikz@ball}{#1}\def\tikz@shading{new ball}\tikz@addmode{\tikz@mode@shadetrue}}

		\newcommand{\ball}[3][white]{%
			\begin{scope}[shift = {(#2,#3)}]
				\ifnum\pdf@strcmp{#1}{white}=\z@%
					\shade[my ball color=#1] (0.15,0.25) circle (0.5);
				\else
					\shade[new ball color=#1] (0.15,0.25) circle (0.5);
				\fi
			\end{scope}
		}
        \newcommand{\urn}[2]{%
			\begin{scope}[shift = {(#1,#2)}]
				\shade[inner color=gray,outer color=white] (0,-0.25) ellipse (1.75 and 0.75);
				\fill[color=white] (0,0) ellipse (1.25 and 0.5);
				\draw (-1.195,-0.15) .. controls (-2,1) and (-0.70,1.75) .. (-1.15,2.95);
				\draw (1.195,-0.15) .. controls (2,1) and (0.70,1.75) .. (1.15,2.95);
				\draw (0,3) ellipse (1.15 and 0.5);
				\draw (0,0) ellipse (1.25 and 0.5);
			\end{scope}
        }
	\makeatother

    \DeclareRobustCommand{\stirling}{\genfrac\{\}{0pt}{}}
    \newmuskip\pFqmuskip

	\newcommand*\pFq[6][8]{%
		\begingroup 
		\pFqmuskip=#1mu\relax
		\mathchardef\normalcomma=\mathcode`,
		\mathcode`\,=\string"8000
		\begingroup\lccode`\~=`\,
		\lowercase{\endgroup\let~}\pFqcomma
		{}_{#2}F_{#3}{\left[\genfrac..{0pt}{}{#4}{#5};#6\right]}%
		\endgroup
	}
	\newcommand{\pFqcomma}{{\normalcomma}\mskip\pFqmuskip}

\begin{document}

\begin{frontmatter}

	\title{On Occupancy Moments\\ and Bloom Filter Efficiency}
	\runtitle{On Occupancy Moments and Bloom Filter Efficiency}

	\begin{aug}
		\author{\snm{Jonathan Burns}\corref{}\ead[label=e1]{jburns@ionic.com}}%
		\runauthor{Burns}
		\affiliation{Ionic Security}
		\address{Ionic Security Inc., 1170 Peachtree St. NE, Suite 400, Atlanta, GA 30309\\ \printead{e1}}
	\end{aug}

    \input{00_abstract.tex}



\end{frontmatter}


	\input{01_intro.tex}
    \input{02_occupy.tex}
    \input{03_bloom.tex}
    \input{04_efficiency.tex}
    \input{05_discussion.tex}
    \input{06_acknowledgements.tex}

    \bibliography{main}
    \bibliographystyle{imsart-number}

\end{document}

%% file: 00_abstract.tex

\begin{abstract}
Two multivariate committee distributions are shown to belong to Berg's family of factorial series distributions and Kemp's family of generalized hypergeometric factorial moment distributions. 
Exact moment formulas, upper and lower bounds, and statistical parameter estimators are provided for the classic occupancy and committee distributions.

The derived moment equations are used to determine exact formulas for the false-positive rate and efficiency of Bloom filters -- probabilistic data structures used to solve the set membership problem.
This study reveals that the conventional Bloom filter analysis overestimates the number of hash functions required to minimize the false-positive rate, and shows that Bloom filter efficiency is monotonic in the number of hash functions.
\end{abstract}

%% file: 01_intro.tex

\section{Introduction}

We consider three occupancy problems where $n$ distinguishable balls are cast into $m$ distinguishable urns:
    \begin{itemize}
		\item \textit{Classic Occupancy Problem}: Cast $n$ balls into $m$ urns. The occupancy number is number of urns containing a ball. \smallskip
        \item \textit{Committee Problem}: A group of $m$ individuals are asked to randomly form $c$ committees of size $k$ where an individual can serve on more than one committee. The occupancy number is the amount of individuals that serve on at least one committee. ($n = c \, k$) \smallskip
        \item \textit{Committee Intersection Problem}: A group of $m$ individuals are randomly selected to serve on committees for $c$ departments where $n_d$ committees of $k_d$ individuals are required for each department $d$. The occupancy number is the amount of individuals serving on a committee from each department. ($n = n_1 k_1 + n_2 k_2 + \cdots + n_c k_c$)
	\end{itemize}

The classic occupancy problem is well-known in probability and combinatorics and there is wealth of literature on the topic; see Feller \cite[Ch. 4, Sec 2]{Feller1968Probability}, Johnson and Kotz \cite[Ch. 3]{Johnson1977Urn}, Kolchin \cite[Ch. 1--3]{Kolchin1978Random}, Johnson, Kemp, and Kotz \cite[Ch. 10, Sec. 4]{Johnson2005Univariate} and Charalambides \cite[Ch. 4]{Charalambides2005Combinatorial}.

Many occupancy models are rooted in biological problems.
The committee problem is a natural generalization of the classic occupancy problem (where $k=1$) and the chromosome problem \cite{Catcheside1946Chromosome,Feller1968Probability} (where $k=2$) to groups of every size $k$.
Mantel and Pasternack \cite{Mantel1968Class} were the first to explicitly state and discuss the committee problem, but it has since gained its own notoriety \cite{Gittelsohn1969Occupancy,Sprott1969Class,White1971Committee,Johnson1977Urn,Walter1979Generalizations}.
The committee problem is closely related to the problem of estimating the population of an organism by repeatedly capturing, tagging, and releasing samples from the population \cite{Berg1974Factorial}, and the multivariate committee model can be used to estimate the enrichment / depletion of genetic samples in an experiment \cite{Kalinka2014Probability} and used to measure molecule similarity for chemical retrieval \cite{Swamidass2007Fingerprint}.

This paper considers two further generalizations of the committee occupancy distribution, and applies this new theory to analyze the efficiency of Bloom filter data compression structures.


\subsection{Bloom filters}
\label{subsec:intro-Bloom}

Given a domain $D$ and a subset $S \subseteq D$, the \textit{set membership problem} seeks to determine if an arbitrary element $x \in D$ belongs to $S$. 
The set membership problem is trivial if either $S$ or $D/S$ can be easily enumerated, such as checking if a letter is a vowel. 
However, many practical applications depend upon computationally inexpensive solutions to the set membership problem, e.g., spell-checking \cite{Bloom1970Space}, validating whether a web crawler has previously visited a website \cite{Olston2010Web}, record deduplication \cite{Lu2012BloomStore}, scanning for computer viruses \cite{Erdogan2005HashAV}, and matching genetic samples to records in large online biological databases \cite{Solomon2015Large}.
A \textit{set membership filter} is a data structure used to efficiently answer the set membership problem.
Simple lists, hash tables, and trees can be used to solve the set membership problem, but succinct representation of a large set requires data compression.
In space or time constrained environments, a solution that occasionally admits false-positive or false-negative results is often acceptable; such filters are called \textit{approximate set membership filters}.

A \textit{Bloom filter} is an approximate set membership filter generated from a superimposed code \cite{Mooers1947Zatocoding,Roberts1979Partial} in the following manner.
The elements of a set are converted into code words using a random hash function $H:D \rightarrow \{0,1\}^m$ and superimposed onto the zero word. Said another way, a Bloom filter is initialized as $\mathcal{F}=0^m$ and elements $s \in S$ are inserted into the filter one at a time by iterating $\mathcal{F} \leftarrow \mathcal{F} \lor H(s)$ where $\lor$ is the bitwise \texttt{OR} operator.
Once all the members of $S$ are superimposed onto $\mathcal{F}$, an element $x \in D$ can be tested for membership in $S$ by checking that $H(x) \land F = H(x)$ where $\land$ is the bitwise \texttt{AND} operator.
False-negative results are impossible since the statement $H(x) \land F \not= H(x)$ implies that $x \not\in S$, but positive results are inconclusive.
Optimization of a Bloom filter involves constructing $H$ to minimize the false-positive rate for the membership test of $\mathcal{F}$.

Bloom filters have inspired a multitude of approximate set-membership filters that are in widespread use across a variety of database and network applications~\cite{Broder2003Applications,Geravand2013Survey,Luo2018Optimizing}, e.g., Apache's Cassandra~\cite{Lakshman2010Cassandra}, RedisLabs's Redis~\cite{Sanfilippo2009Redis}, and Google's Bigtable~\cite{Chang2006Bigtable} distributed storage systems use Bloom filters to perform preliminary cache checks in memory before executing costly disk seeks.
In this paper, we compare two Bloom filter variants that are constructed using different hash functions:
    \begin{itemize}
        \item \textit{Classic Bloom Filter} \cite{Bloom1970Space}: each superimposed code word contains $k$ \textit{distinct} 1-bits and $m\!-\!k$ distinct 0-bits. \smallskip
        \item \textit{Standard Bloom Filter} \cite{Knuth1973Art}: each superimposed code word is generated using $k$ \textit{independent} hash functions that determine the positions of the 1-bits and the remaining bits are 0-bits.
    \end{itemize}
Modeled as a committee problem, each superimposed code word for a classic Bloom filter is a committee selection from a group of individuals.
The corresponding set membership test determines whether the sampled individuals have previously served on a committee.
Similarly, the superimposed code words for a standard Bloom filter correspond to a sample with replacement from the individuals, e.g., choosing committee members by repeatedly drawing names from a hat and replacing the names after selection.
For convenience, both types of Bloom filters are modeled as occupancy problems involving balls and urns throughout the remainder of this paper.

As it turns out, the false-positive rates and efficiency formulas for classic and standard Bloom filters can be derived from the ordinary and binomial moments of the committee occupancy distribution.


\subsection{Contributions}
\label{subsec:contributions}

We show that multivariate versions of the committee distribution belongs to the families of generalized hypergeometric factorial moment distributions (GHFDs) and factorial series distributions (FSDs), and establish simple upper and lower bounds on its moments.

Moment formulas for the multivariate committee distributions are used to derive new equations for the false-positive rate and efficiency of the standard and classic Bloom filters.
Novel results concerning Bloom filter false-positive and efficiency analysis are also shown, including the following observations:
    \begin{itemize}
        \item the common optimization for classic and standard Bloom filters overestimates the optimal number of hash bits to use for small filters;
        \item the efficiency of a standard Bloom filter approaches a limit of $\ln 2$, and decreases as the number of hash bits increase; and
        \item the efficiency of a classic Bloom filter approaches the information-theoretic upper limit, and decreases as the number of hash bits decrease.
    \end{itemize}
The theory presented in this study is also applicable to the wider class of Bloom filter variants used extensively in modern storage and networking systems.


\subsection{Paper Organization}
\label{subsec:organization}

Section~\ref{sec:occupy} concerns the theory of committee occupancy problems.
Sections~\ref{sec:classic-occupancy} and \ref{subsec:committee-occupancy} give a background for the classic occupancy and committee occupancy distributions in terms of finite differences.
Multivariate union and intersection versions of the committee distribution are introduced in Section~\ref{subsec:multivariate-committee}. 
Special attention is given to the single committee intersection problem in Section~\ref{subsubsec:single-committee-intersection}, which is shown to be a GHFMD and a FSD.
Section~\ref{sec:bounds-on-moments} establishes bounds on the moments of committee distributions, and statistical estimators for the parameters of committee models are provided in Section~\ref{subsec:estimators}.

Section~\ref{sec:bloom} applies the theory of committee occupancy problems to standard and classic Bloom filters.
The multivariate union and intersection committee problems translate directly to unions and intersections of Bloom filters; see Section~\ref{sec:union-intersection-bloom-filters} for details.
The false-positive rate of a standard Bloom filter is shown to be a function of raw moments for the classic occupancy distribution in Section~\ref{subsec:standard-false-positive}, and the false-positive rate of the classic Bloom filter is computed as a function of the binomial moments in Section~\ref{subsec:classic-false-positive}.
A surprising amount of literature fails to distinguish between the standard and classic Bloom filter, and the standard Bloom filter is routinely misattributed to Bloom -- who first reported the classic Bloom filter~\cite{Bloom1970Space}.
Section~\ref{sec:historical-notes} attempts to rectify some of the confusion surrounding the classic and standard Bloom filters and the analysis for of their false-positive rates.
Section~\ref{sec:filter-optimization} explores false-positive minimization analytics for the standard and classic Bloom filter, showing that the common Bloom filter configuration requires more computation and yields a higher false-positive rate than the optimal configuration.

The subject of false-positive optimization pivots to filter efficiency in Section~\ref{sec:filter-efficiency}.
Peak efficiencies for standard Bloom filters are shown in Section~\ref{sec:standard-bloom-efficiency} to decrease monotonically as the number of hash bits increase, and peak efficiencies for the classic Bloom filter are shown in Section~\ref{sec:classic-bloom-efficiency} to increase as the number of hash bits increase.
Configurations that yield the poorest efficiencies are characterized in Section~\ref{subsec:valley-efficiency}, and a few brief remarks about compression of the classic Bloom filter are mentioned in Section~\ref{sec:compressed-bloom-efficiency}.

The study is concluded in Section~\ref{sec:discuss} with brief summary, closing remarks, and a discussion of some open conjectures. 

%% file: 02_occupy.tex

\section{Occupancy Problems}
\label{sec:occupy}

Let $\omega \in \Omega$ be an outcome of placing distinguishable balls into $m$ distinguishable urns.
For each $i \in [m]$, let $\chi_i:\Omega \to \mathbb{Z}_2$ be the random variable indicating whether the $i$\textsuperscript{th} urn is occupied:
    \begin{equation}  
    \label{eq:occupancy-event}
    \chi_i(\omega) = 
    	\begin{cases}
    		1, & \mbox{the }i\mbox{\textsuperscript{th} urn contains at least 1 ball,} \\ 
    		0, & \mbox{the }i\mbox{\textsuperscript{th} urn is empty.}
		\end{cases}
    \end{equation}
Throughout this study, urn occupancy is considered an \textit{exchangeable event}, i.e., for any possible choices $1 \le i_1 < i_2 < \dots < i_s$ of $s$ subscripts,
    \begin{align*}
    \mathbb{P}\left[ \chi_{i_1} \, \chi_{i_2} \cdots \chi_{i_s} = 1 \right] = p_s
    \end{align*}
depends on $s$ only and not the actual subscripts $i_t$ ($1 \le t \le s$).
The \textit{occupancy number} is the sum of all urn occupancy random variables
	\begin{align*}
	X = \chi_1 + \chi_2 + \dots + \chi_m \, .
	\end{align*}
It is well known~\cite{Barton1959Contagious,Gittelsohn1969Occupancy,Charalambides2005Combinatorial} that the binomial moments for the occupancy number is given by
    \begin{align}
    \label{eqn:binom-moment}
    \mathbb{E}\left[\binom{X}{r}\right] = \binom{m}{r} \, \mathbb{P} \! \left[ \prod\limits_{l=1}^r \chi_l = 1 \right] \, ,
    \end{align}
which gives the p.m.f.~\cite[p.30--32]{Charalambides2005Combinatorial}
    \begin{align}
    \mathbb{P}\left[ X = i \right] = \binom{m}{i} \sum\limits_{j=0}^{m-i} (-1)^{j} \binom{m-i}{j} \, \mathbb{P} \! \left[ \prod\limits_{l=1}^{i+j} \chi_l = 1 \right] \, .
    \end{align}
For some problems, it is easier to focus on the occupancy of a single urn $\chi_i$ rather than $X$ in aggregate, and equation \eqref{eqn:binom-moment} hints that the binomial moments of an occupancy distribution may have a clean presentation.


\subsection{Classic Occupancy Problem}
\label{sec:classic-occupancy}

The classic occupancy model is the simplest type of occupancy model, and has a few equivalent interpretations as sampling problems:
	\begin{itemize}
		\item \textit{Sampling with replacement} \cite{Johnson1977Urn} -- A single urn initially containing $m$ white balls is sampled one ball at a time, and each sampled ball is replaced by a red ball. The number of red balls in the urn after $n$ balls have been sampled is the occupancy number.
		
		\smallskip
		
        \item \textit{Sampling without replacement} \cite{Kalinka2014Probability} -- Each of $n$ urns contains a set of $m$ balls that are labeled 1 through $m$. A single ball is sampled from each urn, and the number of distinct labels on the sampled balls is equivalent to the occupancy number.
    \end{itemize}
The probability distribution of the classic occupancy number has been rediscovered multiple times and appears in literature under several names: the Stevens-Craig distribution \cite{Stevens1937Significance,Craig1953Utilization}, the Arfwedson distribution \cite{Arfwedson1951Probability}, the coupon collecting distribution \cite{David1962Combinatorial}, and the dixie cup distribution \cite{Johnson1977Urn}.

	\begin{figure}[ht]
    	\centering
		\input{fig/classical-occupancy}
		\caption{One configuration for casting $n=6$ balls into $m=5$ urns that results in $X=3$ occupied urns in the classic occupancy model.}
	\end{figure}
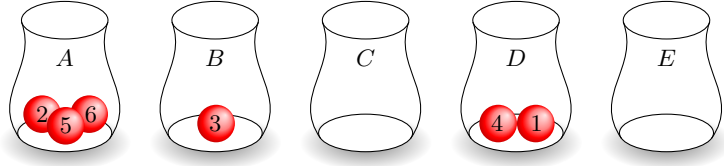

The number of ball placement configuration that result in an occupancy number of $i$ can be enumerated in the following way.
For $1 \le i \le \min\{m,n\}$, there are $\binom{m}{i}$ ways to choose $i$ of the urns to be occupied and $\stirling{n}{i} \, i!$ ways to partition the $n$ balls into the chosen urns where that $\stirling{n}{i}$ denotes a Stirling number of the second kind. 
Hence, the p.m.f. for the occupancy number of the classic occupancy distribution with $n$ balls and $m$ urns is
    \begin{equation} 
    \label{eqn:occupy-probability}
    \mathbb{P}\left[  X = i  \right] = \frac{\stirling{n}{i} \binom{m}{i} \, i!}{m^n} = \stirling{n}{i} \frac{m^{\underline{i}}}{m^n} \, ,
    \end{equation}
and the $r$\textsuperscript{th} moment is
    \begin{equation}
    \label{eq:classic-moments-1}
    \mathbb{E}\left[ X^r \right] = \sum\limits_{i=1}^{m} \stirling{n}{i} \frac{m^{\underline{i}}}{m^n} \, i^r \, .
    \end{equation}
While this expectation formula is sufficient for many applications, it is computationally intractable to obtain high precision calculations for large $m \gg 2^{32}$ on most modern computers. 
Alternatively, the raw moments can be computed from the binomial moments in~\eqref{eqn:binom-moment} using the Stirling inversion:
    \begin{align}
	\label{eq:moment-1}
    \mathbb{E}\left[X^r\right] 
    = \sum\limits_{i=1}^r \stirling{r}{i} \, \mathbb{E}[X^{\underline{i}}]
    = \sum\limits_{i=1}^r \stirling{r}{i} \, m^{\underline{i}} \,\, \mathbb{P} \! \left[ \prod\limits_{j=1}^i \chi_j = 1 \right],
    \end{align}
which is a sum up to $r$ and could be much less than $m$.


\subsubsection{Relation to Finite Differences}
\label{sec:finite-differences}

The moments of occupancy distributions can be expressed compactly via finite difference operators.
Let $\nabla$ be the \textit{backwards difference operator} defined by 
    \begin{equation*}
    \nabla : f(x) \mapsto f(x)-f(x-1)
    \end{equation*}
where the higher order difference is defined recursively by $\nabla^i = \nabla (\nabla^{i-1}) = \nabla^{i-1} ( \nabla )$ for $i>1$.
For example, 
    \begin{align}
    \nabla^i \binom{x}{k} = \nabla^{i-1} \left[ \binom{x}{k} - \binom{x-1}{k} \right] = \nabla^{i-1} \binom{x-1}{k-1} = \cdots = \binom{x-i}{k-i} \, .
    \end{align}
For an arithmetic function $f(x)$, this means $\nabla^i f(x) = \sum_{j=0}^i (-1)^j \binom{i}{j} f(x-j)$.
We adopt the conventional notation 
    $$\nabla [f(x)]_{x=a} = [ f(x) - f(x-1) ]_{x=a} = f(a) - f(a-1)$$ 
to indicate $\nabla f$ evaluated at $a$.
The \textit{forwards difference operator} is similarly defined by 
    $$\Delta : f(x) \mapsto f(x+1)-f(x) \, ,$$ 
and relates to $\nabla f$ through
    \begin{equation*}
    \Delta^i [f(x)]_{x=a} = \nabla^i [f(x)]_{x=a+i} \, .
    \end{equation*}


\begin{lemma}[\cite{Feller1968Probability,Mingshu2011Indentities,Charalambides2005Combinatorial}] 
    \label{lem:classic-moments-2}
    Let $\chi_i$ be i.i.d. urn occupancy random variables under the classic occupancy model with $n$ balls and $m$ urns.
    For $0 \le i \le m$,
    	\begin{align}
        \mathbb{P}\left[ \prod\limits_{j=1}^i \chi_j = 1 \right]
    	    = \sum\limits_{j=0}^i (-1)^j \binom{i}{j} \left( 1 - \frac{j}{m} \right)^n
        	= \frac{\nabla^i [x^n]_{x=m}}{m^n} \, .
    	\end{align}
    The $r$\textsuperscript{th} moment for the occupancy r.v. $X=\sum\chi_j$ is
    	\begin{align} 
    	\label{eqn:classic-moments-2}
        \mathbb{E}\left[X^r\right] = \sum\limits_{i=1}^r \stirling{r}{i} \, \frac{m^{\underline{i}}}{m^n} \, \nabla^i [x^n]_{x=m} \, ,
        \end{align}
    and the binomial and factorial moments satisfy
        \begin{align}
        \label{eqn:binom-fac-moments}
        \frac{\mathbb{E}\left[\binom{X}{r}\right]}{\binom{m}{r}} =
    	\frac{\mathbb{E}[X^{\underline{r}}]}{m^{\underline{r}}} = 
    	\frac{ \nabla^r [x^n]_{x=m}}{m^n} \, .
        \end{align}
\end{lemma}
\begin{proof}
    The probability that the first $i$ urns remains vacant after $n$ balls have been tossed is $\left(\frac{m-i}{m}\right)^n$, and the hypothesis follows from the principle of inclusion-exclusion.
    
    The moments formulas follow from \eqref{eqn:binom-moment} and \eqref{eq:moment-1}.
\end{proof}

Lemma~\ref{lem:classic-moments-2} provides an efficient means for computing high order moments of the occupancy number $X$ when $m$ is large and the power $r$ is small, and comparing \eqref{eq:classic-moments-1} to \eqref{eqn:classic-moments-2} provides a novel proof for the following backwards difference identity reported by Riordan~\cite[p.211]{Riordan1968Combinatorial}.


\begin{corollary}
    \label{cor:stirling-identity}
    For $z \in \mathbb{C}$ and $n \ge 1$ and $r \ge 0$,
        \begin{equation}
        \label{eq:stirling-identity}
        \sum\limits_{i=1}^n \stirling{n}{i} \, \, i^r   z^{\underline{i}} = \sum\limits_{j=0}^r \stirling{r}{j} \,\, \left( \nabla^j  [x^n]_{x=z} \right) \, z^{\underline{j}} \, .
        \end{equation}
\end{corollary}
\begin{proof}
    Fix $n, r \ge 1$ and define the polynomials $f, g :\mathbb{C}\rightarrow\mathbb{C}$ as
        \begin{equation*}
        f(z) = \sum_{i=1}^n \stirling{n}{i} \, i^r z^{\underline{i}} \quad 
        \mbox{and} \quad 
        g(z)=\sum_{j=1}^r \stirling{r}{j} \, z^{\underline{j}} \,\, \nabla^j [x^n]_{x=z} \, .
        \end{equation*}
    For all integers $m \ge n$, the double counting given by \eqref{eq:classic-moments-1} and \eqref{eqn:classic-moments-2} yields
        \begin{equation*}
        f(m) = \sum\limits_{i=0}^n \stirling{n}{i} \, m^{\underline{i}} \,\, i^r 
             = \sum\limits_{i=0}^m \stirling{n}{i} \, m^{\underline{i}} \,\, i^r
             = m^n \, \mathbb{E}[X^r]
             = g(m) \, .
        \end{equation*}
    Since $f$ and $g$ are both polynomials of degree $n$ that agree on at least $n+1$ points, $f=g$ for all $z \in \mathbb{C}$.
    The limit on the RHS of \eqref{eq:stirling-identity} can be extended with $j=0$ to include the well-known Stirling identity $\sum_{i=1}^n \stirling{n}{i} z^{\underline{i}} = z^n$ for $r=0$.
\end{proof}

Prince \cite{Price1946Multinomial} first reported that the p.m.f. for the classic occupancy r.v. satisfies the recursive relation
	\begin{equation}
    \begin{aligned}
    \label{eqn:classic-recurrence}
    & \mathbb{P}\left[ X = i + 1 \mid m, n + 1 \right] = \\
    & \qquad \left( \frac{m - i}{m} \right) \mathbb{P}\left[X = i \mid m, n\right] + \left( \frac{i + 1}{m} \right) \mathbb{P}\left[ X = i + 1 \mid m, n \right] \, .
    \end{aligned}
    \end{equation}
The next theorem shows that the ordinary moments can be expressed recursively too.


\begin{theorem}
    \label{thm:recursive-moments}
    Let $X$ be a classic occupancy r.v. with $n$ balls and $m$ urns.
    For $r \ge 0$, the moments of $X$ satisfy the recurrence relation
        \begin{equation}
        \mathbb{E}\left[X^r \mid m,n \right] = \frac{1}{m} \mathbb{E}\left[X^{r+1} \mid m,n \right] + \left( 1 - \frac{1}{m} \right)^{\!n} \mathbb{E} \left[ X^r \mid m-1, n \right] \, .
        \end{equation}
\end{theorem}
\begin{proof}
    Initially, $E[X^0 \mid m,n] = E[X^0 \mid m-1,n] = 1$ and $\mathbb{E}[X \mid m,n] = m \left(1 - \left( 1 - \frac{1}{m} \right)^n \right)$, and for $r>0$:
    \begin{align*}
    \mathbb{E}\left[ X^r \mid m,n \right] - \frac{1}{m} \mathbb{E} \left[ X^{r+1} \mid m,n \right] & = \frac{1}{m^n} \sum\limits_{i=1}^m \stirling{n}{i}\, m^{\underline{i}} \,\: i^r \left(1-\frac{i}{m}\right) \\
     & = \left( \frac{m-1}{m-1} \right)^n \frac{1}{m^n} \sum\limits_{i=1}^{m-1} \stirling{n}{i}\, (m-1)^{\underline{i}} \, i^r \\
     & = \left( 1 - \frac{1}{m} \right)^n \mathbb{E} \left[ X^{r} \mid m-1,n \right] \, . \qedhere
    \end{align*}
\end{proof}


\subsection{Committee Problems}
\label{subsec:committee-occupancy}

Consider the problem of distributing balls into urns in batches -- where the balls within each batch are cast into distinct urns (see Fig.~\ref{fig:committee-occupancy}).
If a total of $nk$ balls are placed into $m$ urns as $n$ batches of $k$ balls, then the occupancy number is the total number of urns containing at least one ball.
The classic occupancy problem is a special case of the committee problem with $k=1$.

As a committee problem, each batch of balls symbolizes the formation of a committee of size $k$ from a group of $m$ individuals, and urn occupancy represents the number of individuals serving on at least one committee.

	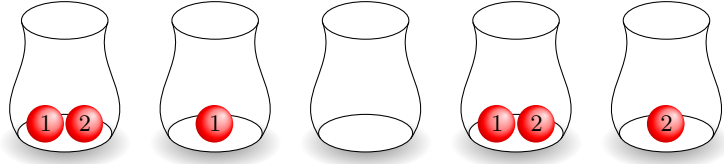
\begin{figure}[ht]
    	\centering
		\input{fig/committee-occupancy}
		\caption{One configuration for casting $n=2$ batches of $k=3$ balls into $m=5$ urns that results in $X=4$ occupied urns in the committee occupancy model.}
		\label{fig:committee-occupancy}
	\end{figure}
    
Some equivalent models for the committee problem include:
	\begin{itemize}
		\item \textit{Sampling with replacement} -- An urn initially containing $m$ white balls is sampled in batches of $k$ balls, and each sample of balls is replaced by $k$ red balls.
        The number of red balls in the urn after $n$ batches have been sampled is the occupancy number.
        
        \smallskip
        
		\item \textit{Sampling without replacement} --  A set of $n$ urns each contain a copy of balls labeled $1,2,\dots,m$. 
        If a sample of $k$ balls are drawn form each of the urns, then the occupancy number is the number balls with distinct labels. 
        
        \smallskip

        \item \textit{Capture-recapture} \cite{Berg1974Factorial} -- A sample of $k$ animals are captured, tagged, and released back into the population of size $m$. 
        Later, another $k$ animals are captured/recaptured and tagged. 
        After $n$ sample trials, the total number of tagged animals is the occupancy number.
        
        \smallskip

		\item \textit{Set union cardinality} \cite{Grandi2015Gamma} -- Let $A$ be a set with $|A|=m$ and $S_1$, $S_2$, $\dots$, $S_n \subseteq A$ be random subsets for which $|S_i|=k$ for all $1 \le i \le n$. 
        The cardinality of the union $\left| \cup_{i\in [n]} S_i \right|$ is the occupancy number.
    \end{itemize}

The p.m.f. for the committee distribution is
	\begin{align}
    \label{eq:batch-probability}
    \mathbb{P}[X = i] & = \binom{m}{i} \sum\limits_{j=0}^{i} (-1)^{i-j} \binom{i}{j} \left[ \frac{\binom{j}{k}}{\binom{m}{k}} \right]^n 
    	= \binom{m}{i} \, \frac{\Delta^i \binom{x}{k}^n_{x=0}}{\binom{m}{k}^n}
    \end{align}
where $m$ is the number of urns, $n$ is the number of batches, and $k$ is the number of balls in each batch \cite{Berg1974Factorial,Johnson1977Urn,Fang1982Finite,Grandi2015Gamma,Mantel1968Class,Roberts1979Partial}.
From \eqref{eq:batch-probability}, the binomial and factorial moments of the committee distribution are given by
    \begin{align}
    \label{eqn:batch-factorial-moments}
    \frac{\mathbb{E}\left[\binom{X}{r}\right]}{\binom{m}{r}}
    = \frac{\mathbb{E}\left[X^{\underline{r}}\right]}{m^{\underline{r}}} 
    = \sum\limits_{i=0}^{r} (-1)^i \binom{r}{i} \left[\frac{\binom{m-i}{k}}{\binom{m}{k}}\right]^n
    = \frac{\nabla^r \binom{x}{k}^n_{x=m}}{\binom{m}{k}^n}
    \end{align}
\cite{Berg1974Factorial,Fang1982Finite,Grandi2015Gamma,Kemp1974Factorial}. 
In particular, the first moment is
    \begin{align}
    \label{eqn:batch-occupancy-mean}
    \mathbb{E}\left[X\right] = m \left( 1 - \left( 1 - \frac{k}{m} \right)^n \right) \, .
    \end{align}

As a suggestion from his referee, Sprott~\cite{Sprott1969Class} generalized the problem from committees of fixed size $k$ to committees with varying sizes $k_1,\dots,k_n$, and obtained
    \begin{align}
    \label{eqn:variable-committee-sizes}
    \mathbb{P}[X=i]=\binom{m}{i} \sum\limits_{j=0}^i (-1)^{i-j} \binom{i}{j} \prod\limits_{d=1}^n \frac{ \binom{j}{k_d} }{ \binom{m}{k_d} } \, ,
    \end{align}
which coincides with the committee union problem in the following section.


\subsection{Multivariate Committee Problems}
\label{subsec:multivariate-committee}

Now suppose that $c$ departments form committees of varying sizes. 
Specifically, let each department $d \in [c]$ select $n_d$ committees of size $k_d$ from a shared pool of $m$ faculty members.
Two natural classes of occupancy problems can be constructed from the unions and intersections of these committees:
    \begin{itemize}
        \item \textit{Committee Union Problem}\footnote{The committee union problem should not be confused with the collective bargaining pursuits of the faculty, which lie beyond the scope of this paper.}: Let $X_{\cup}$ be the r.v. for the occupancy number counting the faculty members which serve on at least one committee from \textit{any} department. \smallskip
        \item \textit{Committee Intersection Problem}: Let $X_{\cap}$ be the r.v. that counts the number of members serving on at least one committee from \textit{every} department.
    \end{itemize}
As an urn model, each department represents a ball coloring from one of $c$ colors.

	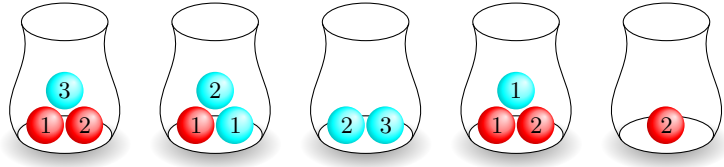
\begin{figure}[ht]
    	\centering
		\input{fig/multi-committee-occupancy}
		\caption{One configuration for casting $n_1=2$ batches of $k_1=3$ red balls and $n_2 = 3$ batches of $k_2 = 3$ cyan balls into $m=5$ urns, resulting in the occupancy number $X_{\cup}=5$ for the committee union problem and $X_{\cap}=3$ for the committee intersection problem.}
	\end{figure}


\subsubsection{Committee Union Problem}
\label{sec:committee-union-problem}

Charalambides~\cite[p.197]{Charalambides2005Combinatorial} considers multiplicities of the variable sized committee problem and gives
    \begin{align}
    \notag \mathbb{P}[X_{\cup} = i \mid n_1, k_1, \dots, n_c, k_c] & = \binom{m}{i} \sum\limits_{j=0}^i (-1)^{i-j} \binom{i}{j} \prod\limits_{d=1}^c \left[ \frac{\binom{j}{k_d}}{\binom{m}{k_d}} \right]^{n_d} \\
    \label{eqn:committee-union-pmf}
        & = \binom{m}{i} \, \Delta^i \left\{ \prod_{d=1}^c \left[ \frac{\binom{x}{k_d}}{\binom{m}{k_d}} \right]^{n_d} \right\}_{x=0}
    \end{align}
with binomial moments
    \begin{equation}
    \label{eqn:committee-union-binom}
    \mathbb{E}\left[\binom{X_{\cup}}{r}\right] = \binom{m}{r} \, \nabla^r \! \left\{ \; \prod\limits_{d=1}^c \left[ \frac{\binom{x}{k_d}}{\binom{m}{k_d}} \right]^{n_d} \right\}_{x=m} \, .
    \end{equation}
Since ball color does not affect occupancy in the committee union problem, the committee sizes $k_1,\dots,k_d$ with multiplicities $n_1,\dots,n_d$ can be re-indexed to $k_1^{\prime},\dots,k_n^{\prime}$ so that
    \begin{align}
    \prod_{d=1}^c \left[ \frac{\binom{x}{k_d}}{\binom{m}{k_d}} \right]^{n_d} = \prod_{l=1}^n \frac{\binom{x}{k_l^{\prime}}}{\binom{m}{k_l^{\prime}}}
    \end{align}
where $n=n_1+\dots+n_c$. 
Said another way: after a suitable re-indexing, the multivariate committee union problem reduces to Sprott's~\cite{Sprott1969Class} generalization with variable committee sizes $k_1,k_2,\dots,k_n$ within a single department \eqref{eqn:variable-committee-sizes}.

As the following theorem shows, the binomial moments and p.m.f. of the committee union distribution can be presented in terms of either finite differences or generalized hypergeometric functions.


\begin{theorem}
    \label{thm:psi-formulas}
    Fix $\vec{k}=(k_1,k_2,\dots,k_n) \in \mathbb{N}^n$ and define
        \begin{align}
        \rho_{\vec{k}}(r,s) := \nabla^r \left[ \prod\limits_{d=1}^n \frac{\binom{x}{k_d}}{\binom{s}{k_d}} \right]_{x=s} .
        \end{align}
    Then $\rho$ satisfies the recursive relation
        \begin{align}
        \rho_{\vec{k}}(r,s) = \rho_{\vec{k}}(r-1,s) - \left(\prod\limits_{d=1}^n 1 - \frac{k_d}{s} \right) \rho_{\vec{k}}(r-1,s-1) \, .
        \end{align}
    with initial condition $\rho_{\vec{k}}(0,s)=1$ for all $s \in \mathbb{N}$, and
        \begin{align}
        \rho_{\vec{k}}(r,s) & = \pFq{n+1}{n}{-r,k_1-s,\dots,k_n-s}{-s, \, \dots \, , \, -s}{1} \, , 
        \end{align}
    for $0 \le \max_{d \in [n]} \, \{k_d\} \le r \le s$ where $_{n+1}F_{n}$ is a generalized hypergeometric function.
\end{theorem}
\begin{proof}
    The initial condition $\rho_{\vec{k}}(0,s)=1$ for all $s\in\mathbb{N}$ is trivial, and the recursive relation can be verified directly:
        \begin{align*}
        \rho_{\vec{k}}(r,s) 
        & = \nabla^r \left[ \prod\limits_{d=1}^{n} \frac{\binom{x}{k_d}}{\binom{s}{k_d}} \right]_{x=s} \\
        & = \nabla^{r-1} \left[ \prod\limits_{d=1}^n \frac{\binom{x}{k_d}}{\binom{s}{k_d}} \right]_{x=s} - \left(\prod\limits_{d=1}^n \frac{\binom{s-1}{k_d}}{\binom{s}{k_d}} \right) \nabla^{r-1} \left[ \prod\limits_{d=1}^n \frac{\binom{x}{k_d}}{\binom{s-1}{k_d}} \right]_{x=s-1} \\
        & = \rho_{\vec{k}}(r-1,s) - \left(\prod\limits_{d=1}^n \frac{s-k_d}{s} \right) \rho_{\vec{k}}(r-1,s-1) \, . 
        \end{align*}
    Expansion of $_{n+1}F_{n}$ gives
        \begin{align*}
        \pFq{n+1}{n}{-r,k_1-m,\dots,k_n-m}{-m, \, \dots \, , \, -m}{1} & = \sum\limits_{j=0}^{\infty} \frac{(-r)^{(j)}}{j!} \prod\limits_{d=1}^n \frac{(k_d - m)^{(j)}}{(-m)^{(j)}} \\
        & = \sum\limits_{j=0}^r (-1)^j \binom{r}{j} \prod\limits_{d=1}^n \frac{ \binom{m-j}{k_d} }{ \binom{m}{k_d} } \\
        & = \nabla^r \left[ \prod\limits_{d=1}^n \frac{\binom{x}{k_d}}{\binom{m}{k_d}} \right]_{x=m}
        \end{align*}
    and completes the proof.
\end{proof}


\begin{corollary}
    For $\max\limits_{j\in [n]} \{k_j \} \le i \le m$,
        \begin{align}
        \label{eqn:committee-union-gh-pmf}
        \mathbb{P}[X_{\cup} = i] = \binom{m}{i} \left( \prod\limits_{d=1}^{n} \frac{\binom{i}{k_d}}{\binom{m}{k_d}} \right) \pFq{n}{n-1}{k_1-i,\dots,k_n-i}{-i,\dots,-i}{1} \, ,
        \end{align}
    and \, $\mathbb{P}[X_{\cup}=i]=0$ \, otherwise, and
        \begin{align}
        \mathbb{E}\left[ \binom{ X_{\cup} }{r} \right] = \binom{m}{r} \, \pFq{n+1}{n}{-r, k_1 \! - \! m,\!\dots,\!k_n \! - \! m}{-m, \, \dots \, , \, -m}{1}
        \end{align}
    for $0 \le r \le m$.
\end{corollary}
\begin{proof}
    Note that
        \begin{align*}
        \Delta^i \left[ \frac{\binom{x}{k_d}}{\binom{m}{k_d}} \right]_{x=0} 
        = \nabla^i \left[ \frac{\binom{x}{k_d}}{\binom{m}{k_d}} \right]_{x=i}
        = \left( \prod\limits_{d=1}^n \frac{\binom{i}{k_d}}{\binom{m}{k_d}} \right) \nabla^{i} \left[  \prod\limits_{d=1}^n \frac{\binom{x}{k_d}}{\binom{i}{k_d}} \right]_{x=i} \, .
        \end{align*}
    Hence, \eqref{eqn:committee-union-pmf} and~\eqref{eqn:committee-union-binom} can be written as
        \begin{align*}
        \mathbb{P}\left[X_{\cup} = i \right] = \binom{m}{i} \left( \prod\limits_{d=1}^n \frac{\binom{i}{k_d}}{\binom{m}{k_d}} \right) \rho(i,i) \quad \mbox{ and } \quad
        \mathbb{E}\left[\binom{X_{\cup}}{r} \right] = \binom{m}{r} \rho(r,m) ,
        \end{align*}
    respectively.
\end{proof}


\begin{corollary}
    \label{cor:committee-union-pgf}
    The p.g.f. for $X_{\cup}$ is
        \begin{align}
        G(z) = z^m \pFq{n}{n-1}{-(m-k_1),\dots,-(m-k_n)}{-m,\dots,-m}{1-\frac{1}{z}} \, .
        \end{align}
\end{corollary}
\begin{proof}
    Follows from Theorem~\ref{thm:ghfd} in Section~\ref{subsubsec:single-committee-intersection}.
\end{proof}

The p.g.f. of a \textit{factorial series distribution} (FSD) is generated from the terms of a factorial series, i.e., the coefficients of the generating function are constructed from a real function $A : [0,m] \to \mathbb{R}^{+}$
    \begin{align}
    \label{eqn:fsd-series}
    A(m) = \sum\limits_{i=0}^m m^{\underline{i}} \, a_i \quad \Longleftrightarrow \quad a_m = \frac{1}{m!} \Delta^m \left[ A(i) \right]_{i=0}
    \end{align}
with $a_i \ge 0$ for $i \in [0, m]$ which do not depend on $m$.
The p.g.f. and p.m.f. of a FSD have forms
    \begin{align}
    G(z) = \sum_{i=0}^m m^{\underline{i}} \, a_i \, z^i
    \qquad \text{and} \qquad 
    \mathbb{P}\left[ X = i \right] = \binom{m}{i} \frac{\Delta^i [A(x)]_{x=0}}{A(m)},
    \end{align}
respectively.


\begin{theorem}[\cite{Berg1974Factorial}]
    \label{}
    $X_{\cup}$ is a FSD with $A(x)=\prod_{d=1}^n \binom{x}{k_d}$.
\end{theorem}

Berg~\cite{Berg1974Factorial} introduced FSDs with regard to the biological capture-recapture problem and reported unbiased parameter estimators for this family of distributions.
Estimates for the number of urns $m$ and the number of balls $n$ based on mean occupancy $\mu$ are provided in Section~\ref{subsec:estimators}.


\subsubsection{Single Committee Intersection Problem}
\label{subsubsec:single-committee-intersection}

In the case where each department nominates a single committee ($n_d=1 \, \,  \forall d \in [c])$, the committee union and intersection problems are complementary.
The following lemma formalizes this claim.


\begin{lemma}
\label{lem:intersect-union-complement}
Let $X_{\cap}$ be a committee intersection r.v. with parameters, $m$, $k_1,\dots,k_c \in \mathbb{N}$ and $n_1=\dots=n_c=1$. Then
    \begin{align}
    \mathbb{P}\left[X_{\cap} = i \mid m, n_d=1, k_d \right] = \mathbb{P}\left[ X_{\cup} = m - i \mid m, n_d = 1, k_d \right]
    \end{align}
for all $0 \le i \le m$.
\end{lemma}

\begin{proof}
Consider De Morgan's law; the intersection of the occupied urns counts the union of unoccupied urns and vice versa.
\end{proof}

\begin{corollary}
The single committee intersection distribution is a FSD.
\end{corollary}


\begin{theorem}
    \label{thm:committee-intersection}
    Let $X_{\cap}$ be a single committee intersection r.v. with parameters, $m$, $k_1,\dots,k_c \in \mathbb{N}$.
    Then the p.m.f. of $X_{\cap}$ is
        \begin{align}
        \label{eqn:committee-intersection-finite-pmf}
        \mathbb{P}\left[ X_{\cap} = i \right] & = \binom{m}{i} \Delta^{m-i} \left\{ \prod\limits_{d=1}^c \frac{\binom{x}{m - k_d}}{\binom{m}{m - k_d}} \right\}_{x=0} \\
        \label{eqn:committee-intersection-gh-pmf}
        & = \binom{m}{i}^{\!\!1-c} \left( \prod\limits_{d=1}^c \binom{k_d}{i} \right) \pFq{c}{c-1}{i-k_1,\dots,i-k_c}{i-m,\dots,i-m}{1}
        \end{align}
    and the $r$\textsuperscript{th} binomial moment is
        \begin{align}
        \label{eqn:committee-intersection-binom-moment}
        \mathbb{E}\left[ \binom{X_{\cap}}{r} \right] & = \binom{m}{r}^{\!\! 1-c} \, \prod\limits_{d=1}^c \binom{k_d}{r} \, .
        \end{align}
\end{theorem}
\begin{proof}
    Using Lemma~\ref{lem:intersect-union-complement}, substitute $i \mapsto m-i$ and $k_d \mapsto m - k_d$ in \eqref{eqn:committee-union-pmf} and \eqref{eqn:committee-union-gh-pmf} to get \eqref{eqn:committee-intersection-finite-pmf} and \eqref{eqn:committee-intersection-gh-pmf}, respectively -- noting that $\binom{i}{k_d}/\binom{m}{k_d} = \binom{m-k_d}{i-k_d}/\binom{m}{i}$ in \eqref{eqn:committee-union-gh-pmf}.
    
    The binomial moments can be established using \eqref{eqn:binom-moment} by noting that 
        \begin{align*}
        \mathbb{P}\left[ \prod_{l=i}^r \chi_l = 1 \right] = \prod_{d=1}^c \frac{ \binom{k_d}{r} }{ \binom{m}{r} } \, .
        \end{align*}
    for each $0 \le r \le m$.
\end{proof}

The family of \textit{generalized hypergeometric factorial-moment distributions} (GHFMDs) are characterized by p.g.f.s of the form $_pF_q\left[(\mathbf{a}); (\mathbf{b}); \lambda (z-1)\right]$, and include the binomial, hypergeometric, and Poisson distributions.
Further details about GHFMDs can be found in \cite{Kemp1974Factorial,Kemp1978Probability,Johnson2005Univariate}.


\begin{theorem}
    \label{thm:ghfd}
    The single committee intersection distribution is a GHFMD.
\end{theorem}
\begin{proof}
    Theorem~\ref{thm:committee-intersection} gives the binomial moment o.g.f. as
        \begin{align*}
        G(t+1) & = \sum\limits_{i=0}^m \binom{m}{i} \frac{\binom{k_1}{i} \cdots \binom{k_c}{i}}{\binom{m}{i} \cdots \binom{m}{i}} t^i 
        = \pFq{c}{c-1}{-k_1, \dots, -k_c}{-m,\dots,-m}{-t}
        \end{align*}
    which is also the factorial moment e.g.f.; hence, the ordinary p.g.f. is found by setting $t=z-1$:
        \begin{align}
        \label{eqn:unoccupied-pgf}
        G(z) = \pFq{c}{c-1}{-k_1,\!\dots,-k_c}{-m, \, \dots \, , \, -m}{1-z}
        \end{align}
    which fits the form of a GHFMD.
\end{proof}

\begin{proof}[Proof of Corollary~\ref{cor:committee-union-pgf}]
    Applying Lemma~\ref{lem:intersect-union-complement} to \eqref{eqn:unoccupied-pgf}, the p.g.f. for the unoccupied urns $U_{\cup} = m - X_{\cup}$ is
        \begin{align*}
        G(z) = \pFq{c}{c-1}{-(m-k_1),\dots,-(m-k_c)}{-m,\dots,-m}{1-z} \, .
        \end{align*}
    The p.g.f. of $X_{\cup}$ can be obtained through $\mathbb{P}[X_{\cup}=i] = \mathbb{P}[U_{\cup}=m-i]$, which amounts to reversing the order of the power series and shifting the power of the base by $m$, i.e., setting $z \mapsto z^{-1}$ and multiplying $G(z^{-1})$ by $z^m$.
\end{proof}

Kemp and Kemp~\cite{Kemp1974Factorial} use the differential equation for generalized hypergeometric functions to express the probabilities, raw moments, and factorial moments of GHFMDs in terms of recurrence relations.
Kumar~\cite{Kumar2009EGHPD} also provides an alternative recursive relation for the broader class of \textit{extended generalized hypergeometric probability distributions} (EGHPDs) which contains the GHFMDs.
Application of Kemp's method yields the following recursive relations.


\begin{corollary}
    Let $X_{\cap}$ be a committee intersection r.v. with $n_d=1$ committees of sizes $k_d$ for each department $d \in [c]$ drawn from a pool of $m$ faculty members. 
    Then
        \begin{align}
        \label{eqn:recursive-pmf}
        \mathbb{P}\left[X_{\cap} = i \right] &= \sum\limits_{j=1}^{c-1} (-1)^j \binom{i+j}{j} Q_{m,\vec{k}}(i,j) \, \mathbb{P}\left[X_{\cap} = i + j \right]
        \end{align}
    where
        \begin{align*}
        Q_{m,\vec{k}}(i,j) & = \sum\limits_{u=j}^c (-1)^{c-u} \sum\limits_{v=j}^u \stirling{u}{v} \binom{i}{v\!-\!j} (v\!-\!1)! \\
        & \qquad \times \left[ j \binom{c\!-\!1}{u\!-\!1} (m\!+\!1)^{c-u} - v \, e_{c-u} (\vec{k}) \right] \bigg/ \prod\limits_{l=1}^c (i-k_l)
        \end{align*}
    such that $e_{s}(\vec{k})$ is an elementary symmetric polynomial in $k_1,\dots,k_c$.
    The factorial and binomial moments can be expressed recursively as
        \begin{align*}
        \mu_{[r+1]} = (m-r) \left( \prod_{j=1}^c \frac{k_j - r}{m-r} \right) \mu_{[r]} \, \mbox{ and } \,
        b_{(r+1)} = \frac{m-r}{r+1} \left( \prod_{j=1}^c \frac{k_j - r}{m-r} \right) b_{(r)} \, .
        \end{align*}
\end{corollary}
\begin{proof}
    Since $X_{\cap}$ is a GHFMD, the f.m.g.f. is a generalized hypergeometric function, so it is a solution to the differential equation
        \begin{align}
        \label{eqn:ghfmd-diff-eqn}
        \vartheta \left( \vartheta  - (m + 1) \right)^{c-1} G(1 + t) = -t \prod_{j=1}^{c} \left(\vartheta - k_j \right) G(1 + t)
        \end{align}
    where $\vartheta$ is the differential operator $\vartheta = t D = t \frac{d}{dt}$.
    Expanding the functional equation in terms of $D$ and the fact that $G(1+t)$ is the o.g.f. for the binomial moments leads to the following recurrence
        \begin{align*}
        r \, (r - (m+1))^{c-1} \, b_{(r)} = - \left( \prod_{j=1}^c (r-1) - k_j \right) b_{(r-1)} \, ,
        \end{align*}
    i.e.,
        \begin{align*}
        b_{(r)} = \frac{m - (r-1)}{(r-1)+1} \left( \prod_{j=1}^c \frac{ k_j - (r-1) }{ m - (r-1) } \right) b_{(r-1)}
        \end{align*}
    which also concurs with \eqref{eqn:committee-intersection-binom-moment}.
    
    Equation~\eqref{eqn:recursive-pmf} can be obtained in a similar way by substituting $t=z-1$ in \eqref{eqn:ghfmd-diff-eqn} and equating the coefficients of the resulting polynomial.
\end{proof}


\subsubsection{General Committee Intersection Problem}
\label{subsubsec:general-committee-intersection}

Nishimura and Sibuya \cite{Nishimura1988Occupancy} consider the case of the committee intersection problem where $c=2$ and $k_1 = k_2 = 1$ but $n_1$ and $n_2$ are free, and note that the factorial moments can be expressed compactly as
	\begin{equation}
	\label{eqn:two-color-factorial-moments}
	\mathbb{E}[{X_{\cap}}^{\underline{r}} \mid m, n_1, n_2 ] = \frac{m^{\underline{r}}}{m^{n_1+n_2}} \nabla^r m^{n_1} \, \nabla^r m^{n_2} \, ,
    \end{equation}
i.e., the binomial moments are 
    $$\mathbb{E}\left[\binom{X_{\cap}}{r} \mid m, n_1, n_2 \right] = \binom{m}{r} \frac{\nabla^r m^{n_1}}{m^{n_1}} \frac{\nabla^r m^{n_2}}{m^{n_2}}$$
which is shown in the following theorem to directly generalize to arbitrary department and committee sizes.


\begin{theorem}
    \label{thm:mutual-batch-occupy-moment}
    Let $X_{\cap}$ be a committee intersection occupancy r.v. with $c$ departments having $n_d$ committees of size $k_d$ for each department $d \in [c]$. Then
        \begin{align}
        \mathbb{P}\left[X_{\cap} = i \right] & = \binom{m}{i} \Delta^{m-i} \left\{ \prod\limits_{d=1}^c \nabla^{m-y} \left[ \frac{\binom{x}{k_d}^{n_d} }{\binom{m}{k_d}^{n_d}} \right]_{x=m}  \right\}_{y=0} \\
        & = \binom{m}{i} \sum_{j=0}^{m-i} (-1)^j \binom{m-i}{j} \prod\limits_{d=1}^c \sum\limits_{l=0}^{i+j} (-1)^{i+j-l} \binom{i+j}{l} \left( \frac{ \binom{m-l}{k_d} }{ \binom{m}{k_d} } \right)^{\!\!n_d}
        \end{align}
    with binomial moments
        \begin{align}
        \label{eq:mutual-batch-expected}
            \mathbb{E}\left[ \binom{X_{\cap}}{r} \right] = \binom{m}{r} \, \prod\limits_{d=1}^{c}  \nabla^r \! \left\{ \left[ \frac{ \binom{x}{k_d}}{\binom{m}{k_d}} \right]^{n_d} \right\}_{x=m} \, .
        \end{align}
\end{theorem}
\begin{proof}
    Equation~\eqref{eqn:binom-moment} gives
        \begin{align*}
    	\mathbb{E}\left[ {X_{\cap}}^r \right] = \sum\limits_{i=1}^{r} \stirling{r}{i} \, m^{\underline{i}} \,\: \mathbb{P} \! \left[ \prod\limits_{j=1}^{i}  \chi_j = 1 \right] \label{eqn:multinomial},
        \end{align*}
    and the product $\prod_j \chi_j = 1$ indicates that the first $i$ urns are occupancy by all $c$ colors. 
    Since the balls of each color are distributed independently, Lemma~\ref{lem:classic-moments-2} and \eqref{eqn:batch-factorial-moments} can be applied to a r.v. $X_d$ each color $d \in [c]$ as
    	\begin{equation*}
    	\mathbb{P} \! \left[ \prod\limits_{j=1}^i \chi_j = 1 \right] 
    	    = \prod\limits_{d=1}^c \mathbb{P} \! \left[ \prod\limits_{j=1}^i \chi_j = 1 \right] \!
    	    = \prod\limits_{d=1}^c \frac{\mathbb{E}\!\left[\binom{X_d}{i}\right]}{\binom{m}{i}} 
    	    = \prod\limits_{d=1}^c \frac{\nabla^i \binom{m}{k_d}^{n_d}}{\binom{m}{k_d}^{n_d}}
        \end{equation*}
    which yields \eqref{eq:mutual-batch-expected}, and the p.m.f. can be computed from the binomial moments.
\end{proof}


\begin{corollary}
    Let $X_{\cap}$ be a committee intersection r.v. with $c$ departments that each select committees of varying sizes $k_{d,1},\dots,k_{d,n_d}$. Then
        \begin{align}
        \mathbb{P}\left[X_{\cap} = i \right] = \binom{m}{i} \Delta^{m-i} \left\{ \prod\limits_{d=1}^c \nabla^{m-y} \left[ \prod\limits_{l=1}^{n_d} \frac{\binom{x}{k_{d,l}} }{\binom{m}{k_{d,l}}} \right]_{x=m} \right\}_{y=0}
        \end{align}
    with binomial moments
        \begin{align}
            \mathbb{E}\left[ \binom{X_{\cap}}{r} \right] = \binom{m}{r} \, \prod\limits_{d=1}^{c}  \nabla^r \! \left\{ \prod\limits_{l=1}^{n_d} \frac{ \binom{x}{k_{d,l}}}{\binom{m}{k_{d,l}}} \right\}_{x=m} \, .
        \end{align}
\end{corollary}


\subsection{Occupancy Bounds and Statistics}
\label{sec:approximations}

The next two subsections provide some approximations and statistical tools for working with committee distributions, and help to develop a statistical model for Bloom filters and their associated false-positive rates.


\subsubsection{Bounds on Moments}
\label{sec:bounds-on-moments}

The following theorem is useful for establishing some simple bounds on the moments of committee problems.


\begin{theorem}
    \label{thm:batch-occupancy-bounds}
    Fix $m \in \mathbb{N}^+$ and $k,n \in \mathbb{N}^0$. 
    For $0 \le r \le \min\{k, m - k\} \le m$,
        \begin{equation*}
        \left( \frac{\mu}{m} \right)^r \ge \frac{ \nabla^r \! \left[ \binom{x}{k}^{n} \right]_{x=m} }{ \binom{m}{k}^{n} } \ge \frac{ \binom{ \mu }{r} }{ \binom{m}{r} } \ge \frac{\nabla^r \! \left[ x^{n k} \right]_{x=m}}{m^{n k}} 
        \end{equation*}
    where $\mu = m \left( 1 - \left( 1 - \frac{k}{m} \right )^n \right)$ and $\binom{x}{y}= x^{\underline{y}} \, / y! \,$ for $x \in \mathbb{R}$ and $y\in \mathbb{N}$.
\end{theorem}
\begin{proof}
    For the uppermost bound, note that for all $0 \le r <  m-k$,
        \begin{align*}
        \nabla^r \! \left[ \binom{x}{k}^n \right]_{x=m} 
            & = \nabla^{r-1} \! \left[ \binom{x}{k}^n - \binom{x-1}{k}^n \right]_{x=m} \\
            & = \nabla^{r-1} \! \left[ \binom{x}{k}^n \left( 1 - \left( 1 - \frac{k}{x} \right)^n \right) \right]_{x=m} \\
            & \le \frac{\mu}{m} \, \nabla^{r-1} \! \left[ \binom{x}{k}^n \right]_{x=m} \, ,
        \end{align*}
    since $0 \le \frac{k}{x} \le 1$ for all $m - r \le x \le m$. 
    Iterating the above inequality $r$ times gives
        \begin{align*}
        \nabla^r \! \left[ \binom{x}{k}^n \right]_{x=m} 
            \le \left( \frac{\mu}{m} \right)^r \nabla^{0} \! \left[ \binom{x}{k}^n \right]_{x=m} 
            =   \left( \frac{\mu}{m} \right)^r \binom{m}{k}^n \, ,
        \end{align*}
    i.e., the upper bound holds.
    
    For the middle inequality, let $X$ be the occupancy number of a committee random variable with $m$ urns and $n$ batches of $k$ balls to represent each committee selection.
    From \eqref{eqn:batch-factorial-moments}, the mean is $\mu = \mathbb{E}[X] = m \left( 1 - \left(1-\frac{k}{m}\right)^{n} \right)$.
    Once the first committee is chosen -- i.e., a batch is cast -- the urn occupancy is $k$ and monotonically increases as more batches are cast, which means $X \ge k \ge r$.
    The falling factorial function $x^{\underline{r}}$ is convex for $x \ge r$, so Jensen's inequality gives
        \begin{equation*}
        \frac{\nabla^r \! \left[ \binom{x}{k}^n \right]_{x=m}}{\binom{m}{k}^n} 
            = \frac{\mathbb{E}[X^{\underline{r}}]}{m^{\underline{r}}}
            \ge \frac{ \mathbb{E}[X]^{\underline{r}} }{m^{\underline{r}}} = \frac{ \binom{\mu}{r} }{ \binom{m}{r}  } \, .
        \end{equation*}

    The rightmost inequality of the hypothesis can be proved inductively.
    The initial case $r=0$ holds $\binom{\mu}{0} / \binom{m}{0} = 1 = m^{nk} / m^{nk}$.
    If the inductive hypothesis holds for $r \le i$, then
        \begin{align*}
        \frac{ \binom{\mu}{i+1} }{ \binom{m}{i+1} } 
            & = \left( \frac{\mu - i}{m - i} \right) \frac{ \binom{\mu}{i} }{ \binom{m}{i} } \\
            & = \left( 1 - \frac{m}{m-i} \left( \frac{m-k}{m} \right)^n \right) \frac{ \binom{\mu}{i} }{ \binom{m}{i} } \\
            & \ge \left( 1 - \left( 1 - \frac{k}{m} \right)^{n-1} \right) \frac{ \binom{\mu}{i} }{ \binom{m}{i} } \\
            & \ge \left( 1 - \left( 1 - \frac{1}{m} \right)^{k n} \right) \frac{ \binom{\mu}{i} }{ \binom{m}{i} } 
        \end{align*}
    which means that
        \begin{align*}
        \left( 1 - \left( 1 - \frac{1}{m} \right)^{k n} \right) \frac{ \binom{\mu}{i} }{ \binom{m}{i} }
            & \ge \left( 1 - \left( 1 - \frac{1}{m} \right)^{k n} \right) \frac{\nabla^i \! \left[ x^{n k} \right]_{x=m}}{m^{n k}} \\
            & \ge \frac{1}{m^{n k}} \nabla^i \! \left[ x^{n k} \left(1 - \left( 1 - \frac{1}{x} \right)^{n k} \right) \right]_{x=m} \\
            & = \frac{1}{m^{n k}} \nabla^{i+1} \! \left[ x^{n k} \right]_{x=m}
        \end{align*}
    and completes the proof.
\end{proof}


\subsubsection{Occupancy Statistics}
\label{subsec:estimators}

The moment formulas for the occupancy numbers can be used to estimate unknown urn model parameters. 
For example, in the capture-recapture model, the number of tagged animals can help to estimate the total population size \cite{Berg1974Factorial}.
As is discussed later in Section~\ref{sec:bloom}, the bits sum of a Bloom filter can be used to estimate the number of items it stores.


\begin{lemma}[\cite{Berg1974Factorial,Grandi2015Gamma,Grandi2018Analysis}]
    \label{lem:batch-mean-and-variance}
    Let $\mu$ and $\sigma^2$ be the mean and variance of the batch occupancy number $X$ with $m$ urns and $k$ batches of $n$ balls. 
    Then
    	\begin{equation}
        \label{eq:batch-mean} 
        \mu = m \left( 1 - \left( 1-\frac{k}{m} \right)^n \right)
        \end{equation}
    and
    	\begin{equation}
        \label{eq:batch-variance} 
        \sigma^2 = m \! \left[ 1 \! - \! \left( 1 - \frac{k}{m} \right)^{\!\!n} \right] \! \left[ 1 \! - \! m \! \left( \! 1 \! - \! \left( \! 1 \! - \! \frac{k}{m} \right)^{\!\!n} \right) \! + \! (m \! - \! 1) \! \left( \! 1 \! - \! \frac{k}{m \! - \! 1} \right)^{\!\!n} \right] \! .
    	\end{equation}
\end{lemma}


\begin{corollary}[\cite{Price1946Multinomial,Johnson2005Univariate,Grandi2018Analysis}]
    \label{cor:classic-mean-and-variance}
    Let $\mu$ and $\sigma^2$ be the mean and variance of the classical occupancy number $X$ with $m$ urns and $n$ balls. Then
    	\begin{equation}
        \label{eq:classic-mean}
        \mu = m \left( 1 - \left( 1-\frac{1}{m} \right)^n \right)
        \end{equation}
    and
    	\begin{equation}
        \label{eq:classic-variance}
        \sigma^2 = m \left( \! \left( 1 - \frac{1}{m} \right)^{\!\!n} \!\! - \left(1 - \frac{2}{m} \right)^{\!\!n} \right) - m^2 \left( \! \left( 1 - \frac{1}{m} \right)^{\!\!2n} \!\! - \left(1 - \frac{2}{m} \right)^{\!\!n} \right) .
    	\end{equation}
\end{corollary}


\begin{corollary}[\cite{Price1946Multinomial,Swamidass2007Fingerprint}]
    \label{cor:method-of-moments}
    Let $\mu$ be the occupancy number obtained from a batch occupancy r.v. with $m$ urns and $n$ batches of $k$ balls. 
    The method of moments and ML estimators for $n$ are
        \begin{equation} 
        \label{eq:moments-n}
    	\hat{n} = \frac{\ln \left(1-\frac{\mu}{m} \right)}{\ln \left(1-\frac{k}{m} \right)} \, .
        \end{equation}
\end{corollary}

\begin{proof}
    Given \eqref{eq:classic-mean}, $\ln \left( 1 - \frac{\mu}{m} \right) = \hat{n} \, \ln \left( 1 - \frac{k}{m} \right)$.
\end{proof}


\begin{lemma}[\cite{Berg1974Factorial,Johnson1977Urn}]
    Let $\mu$ be the occupancy number obtained from a batch occupancy r.v. with $m$ urns and $n$ batches of $k$ balls.
    The MVUE for $m \le n k$ is
        \begin{equation}
        \hat{m} = \mu \left( 1 + \frac{ \Delta^{\mu-1} \binom{x}{k}^n_{x=0} }{ \Delta^{\mu} \binom{x}{k}^n_{x=0} } \right) .
        \end{equation}
\end{lemma}


\begin{corollary}[\cite{Berg1974Factorial,Charalambides2005Combinatorial}]
    Let $\mu$ be the occupancy number obtained from a classic occupancy r.v. with $n$ balls and $m$ urns. The MVUE for $m$ is 
        \begin{equation}
        \hat{m} = \begin{cases}
            \mu + \frac{\stirling{n}{\mu-1}}{\stirling{n}{\mu}} & \mbox{for } m > n \\
            \frac{ \stirling{n+1}{\mu} }{ \stirling{n}{\mu} } & \mbox{for } m \le n 
        \end{cases} \, .
        \end{equation}
\end{corollary}

%% file: fig/classical-occupancy.tex
\begin{tikzpicture}[scale=0.5]

    \urn{0}{0}
	\ball[red]{-0.75}{0.25}
    \node at (-0.6,0.5) {2};
	\ball[red]{0.5}{0.25}
    \node at (0.7,0.5) {6};
	\ball[red]{-0.125}{-0.06}
    \node at (0.03,0.21) {5};
    \node at (0,2) {$A$};

    \urn{4}{0}
	\ball[red]{3.875}{0}
    \node at (4,0.25) {3};
    \node at (4,2) {$B$};

    \urn{8}{0}
    \node at (8,2) {$C$};

    \urn{12}{0}
	\ball[red]{11.375}{0}
    \node at (11.50,0.25) {4};
	\ball[red]{12.375}{0}
    \node at (12.55,0.25) {1};
    \node at (12,2) {$D$};

    \urn{16}{0}
    \node at (16,2) {$E$};

\end{tikzpicture}

%% file: fig/committee-occupancy.tex
\begin{tikzpicture}[scale=0.5]

    \urn{0}{0}
	\ball[red]{-0.67}{0}
    \node at (-0.5,0.25) {1};
	\ball[red]{0.37}{0}
    \node at (0.54,0.25) {2};

    \urn{4}{0}
	\ball[red]{3.83}{0}
    \node at (4,0.25) {1};

    \urn{8}{0}

    \urn{12}{0}
	\ball[red]{11.33}{0}
    \node at (11.50,0.25) {1};
	\ball[red]{12.38}{0}
    \node at (12.55,0.25) {2};

    \urn{16}{0}
	\ball[red]{15.83}{0}
    \node at (16,0.25) {2};

\end{tikzpicture}

%% file: fig/multi-committee-occupancy.tex
\begin{tikzpicture}[scale=0.5]

    \urn{0}{0}
	\ball[red]{-0.67}{0}
    \node at (-0.5,0.25) {1};
	\ball[red]{0.37}{0}
    \node at (0.54,0.25) {2};
    \ball[cyan]{-0.15}{0.92}
    \node at (0,1.17) {3};

    \urn{4}{0}
	\ball[red]{3.33}{0}
    \node at (3.5,0.25) {1};
	\ball[cyan]{4.37}{0}
    \node at (4.54,0.25) {1};
    \ball[cyan]{3.85}{0.92}
    \node at (4,1.17) {2};

    \urn{8}{0}
	\ball[cyan]{7.33}{0}
    \node at (7.5,0.25) {2};
	\ball[cyan]{8.37}{0}
    \node at (8.54,0.25) {3};

    \urn{12}{0}
	\ball[red]{11.33}{0}
    \node at (11.5,0.25) {1};
	\ball[red]{12.37}{0}
    \node at (12.54,0.25) {2};
    \ball[cyan]{11.85}{0.92}
    \node at (12,1.17) {1};

    \urn{16}{0}
	\ball[red]{15.83}{0}
    \node at (16,0.25) {2};

\end{tikzpicture}

%% file: 03_bloom.tex

\section{Bloom Filters}
\label{sec:bloom}

Bloom~\cite{Bloom1970Space} provides a method for using a hash function $H: D \rightarrow [m]^k$ to generate a set membership filter for $S = \{x_1$, $x_2$, $\dots$, $x_n \} \subseteq D$ in the following way.
A Bloom filter is initialized as an $m$-bit word (or array) of all zeros
    \begin{equation*}
    B = b_1 \, b_2 \dots b_m = 0 \, 0 \dots 0 
    \end{equation*}
and the set $S$ is encoded into the filter by computing $k$ bit positions of $B$ for each $x_i \in S$ using a hash function
    \begin{equation*}
    H(x_i) = \left\{ h_1(x_i), \, h_2(x_i), \dots, h_k(x_i) \right\}
    \end{equation*}
and setting $b_d = 1$, if $h_j (x_i) = d$ for some $j \in [k]$ and $x_i \in S$.
Any $y \in D$ can be tested for membership in $S$ by checking that all the hash bits
    $$b_{h_1(y)} = b_{h_2(y)} = \cdots = b_{h_k(y)} = 1 \, .$$
If $b_{h_j(y)}=0$ for any $j \in [k]$, then the membership test guarantees that $y \not\in S$.
A positive result is inconclusive, since the membership check for a Bloom filter can yield a false-positive result (see Fig~\ref{fig:bloom-filter}).

	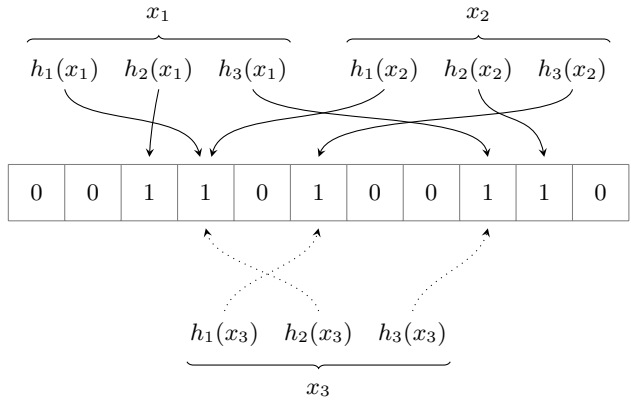
\begin{figure}[ht]
    	\centering
		\input{fig/bloom-false-positive}
		\caption{A Bloom filter $B$ of length $m=11$ encoding $S=\{x_1,x_2\}$ with $H=\{h_1,h_2,h_3\}$, i.e., $n=|S|=2$ and $k=|H|=3$. %
        The membership check for $x_3 \in D$ yields a false-positive result.}
    	\label{fig:bloom-filter}
	\end{figure}

Intuitively, the false-positive rate for a filter can be decreased by either storing fewer items $n$ or extending the filter length $m$.
Increasing the number of hash bits per item $k$ initially decreases the false-positive rate for Bloom filter with fixed $m$ and $n$ before logistically increasing to 1; see Fig.~\ref{fig:fp-vs-k}.
Exact formulas for the false-positive rates of classic and standard Bloom filters in terms of $m$, $n$, and $k$ are provided in Section~\ref{subsec:false-positive}.

Bloom's method for filter construction \cite{Bloom1970Space} uses a hash function $H:D \rightarrow \binom{[m]}{k}$ that outputs $k$ distinct integers belonging to $[m]$.
A variant of Bloom's method that uses $k$ independent hash functions $h_1,\dots,h_k$ where $h_i:D \rightarrow [m]$ to define $H$ has become the de facto approximate set membership filter, and is often misattributed to Bloom. 
See Section~\ref{sec:historical-notes} for further historical details.
In \cite{Grandi2018Analysis}, the approximate set membership filters that encode items using $H_C:D \rightarrow \binom{[m]}{k}$ are called \textit{classic Bloom filters} and filters that use $H_S:D \rightarrow [m]^k$ are called \textit{standard Bloom filters}. 
We adopt this notation, but note that both the terms ``classic'' \cite{Christensen2010new,Fan2014Cuckoo,Peng2018Persistent} and ``standard'' \cite{Bonomi06dLeft,Kirsch2008Less,Peng2018Persistent,Weaver2018XorSat} are often used interchangeably to describe the standard Bloom filter.

The dynamics of inserting items from a random set $S$ into a $m$-bit Bloom filter can be simulated with an urn model for the committee union problem with $nk$ balls and $m$ urns.
A classic Bloom filter uses $H$ to generate $k$ distinct bit positions for each item, which corresponds to a committees of size $k$ (batches of $k$ balls). 
For a standard Bloom filter, the $k$ hashes from $H$ are independent so the $n$ items of $S$ generate a total of $nk$ independent balls corresponding to $nk$ committees with a single member.

The bit value of the filter $B$ at each position $1 \le i \le m$ corresponds to the event $\chi_i(\omega)$, and the bit sum $|B| = \sum_{i=1}^m b_i$ is the occupancy number $X$.
Thus, both classic and standard Bloom filters can be modeled as committee union problems; see Table~\ref{tab:bloom-as-occupancy} for a side-by-side comparison.

    \begin{table}[t]
        \centering
        \begin{tabular}{|l|c|c|c|}
        \cline{3-4}
        \multicolumn{2}{c|}{ }           & \multicolumn{2}{c|}{Bloom filter} \\
        \cline{3-4}
        \multicolumn{2}{c|}{ }       & Classic   & Standard    \\
        \hline
        \multirow{6}{*}{\begin{turn}{90}Urn Model\end{turn}} 
            & \multirow{2}{*}{Balls}    & \multirow{2}{*}{$n$}  & \multirow{2}{*}{$1$}      \\
            &                           &                       &                           \\
            & \multirow{2}{*}{Batches}  & \multirow{2}{*}{$k$}  & \multirow{2}{*}{$nk$}     \\
            &                           &                       &                           \\
            & \multirow{2}{*}{Urns}     & \multirow{2}{*}{$m$}  & \multirow{2}{*}{$m$}      \\
            &                           &                       &                           \\
        \hline
        \multirow{6}{*}{\begin{turn}{90}Occupancy\end{turn}}
            & \multirow{3}{*}{ $\displaystyle \mathbb{P}(X=i)$ }  
            & \multirow{3}{*}{ $\displaystyle \binom{m}{i} \frac{\Delta^i \binom{x}{k}^{n}_{x=0}}{\binom{m}{k}^n}$ } 
            & \multirow{3}{*}{ $\displaystyle \stirling{nk}{i} \frac{m^{\underline{i}}}{m^{nk}}$ }      \\
            &                           &                       &                           \\
            &                           &                       &                           \\
            & \multirow{3}{*}{ $\displaystyle \mathbb{E}(X)$ }
            & \multirow{3}{*}{ $\displaystyle m \left( 1 - \left(1 - \frac{k}{m} \right)^{n} \right)$ }
            & \multirow{3}{*}{ $\displaystyle m \left( 1 - \left( 1 - \frac{1}{m} \right)^{n k} \right)$ }    \\
            &                           &                       &                           \\
            &                           &                       &                           \\
        \hline
        \end{tabular}
        \caption{Classic and standard Bloom filters modeled as committee occupancy problems.}
        \label{tab:bloom-as-occupancy}
    \end{table}

For a standard Bloom filter, the probability of encoding a collision-free code word is
    \begin{equation*}
    \mathbb{P}_S \left( X=k \mid n=1 \right) = \frac{ m^{\underline{k}} }{ m^k } \quad \mbox{and} \quad \lim_{m \rightarrow \infty} \frac{m^{\underline{k}}}{m^k} = 1 \, .
    \end{equation*}
Hence for fixed $k$ and increasing $m \gg k$, the asymptotic behavior of the standard Bloom filter is equivalent to the classic Bloom filter.


\subsection{Union and Intersection of Bloom Filters}
\label{sec:union-intersection-bloom-filters}

In some constrained or distributed applications, it is more practical to generate a collection of Bloom filters in parallel and aggregate the results than it is to serially construct a single global Bloom filter.
For example, a group of data storage centers could generate individual Bloom filters for the list of file IDs of the documents stored at each location.
When a user requests a file from a data center, the server can check each Bloom filter for the requested file ID to find the data center that hosts the document.
The individual Bloom filters can be combined to form a master Bloom filter for all documents, or the intersection of the filters can be analyzed to detect file duplication across the data centers.

Let $B_1$, $B_2$, \dots, $B_c$ be $m$-bit Bloom filters that use a common hash function $H$ to encode the sets $S_1, S_2,\dots, S_c$, respectively.
A new filter $B_{\lor}$ for $\bigcup_{i=1}^c S_i$ can be constructed directly from $B_1,B_2,\dots,B_c$ by computing the bitwise \texttt{OR} operation for each of the $m$ positions in the filter; see Fig~\ref{fig:bloom-int-union}.
The resulting $B_{\lor}$ can be modeled as a committee problem with the same $m$ and $k$ as the starting filters and $n =\sum_{i=1}^c n_i$ aggregate items.
Any $x$ that tests positive for a $B_i$ also tests positive for $B_{\lor}$.
Unfortunately, this means $B_{\lor}$ preserves all the false-positive results from the input filters and potentially introduces new false-positives that would register as true-negatives for each of the original filters. 
Therefore construction of $B_{\lor}$ is lossy form of compression.
    
    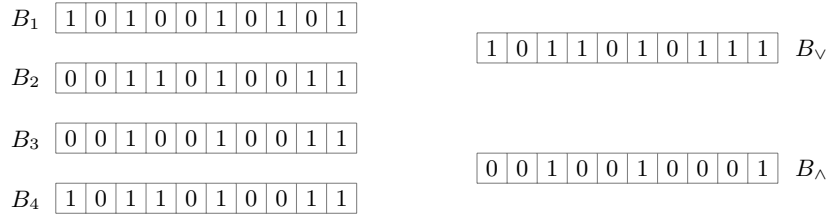
\begin{figure}[t]
        \centering
    	\input{fig/bloom_union_intersection}
        \caption{Unions and intersections of Bloom filters are Bloom filters and can be modeled as committee occupancy problems.}
        \label{fig:bloom-int-union}
    \end{figure}

Similarly, intersecting Bloom filters generates another aggregate Bloom filter which can be modeled as a committee intersection problem.


\begin{corollary}
    \label{cor:filter-intersection}
    Let $B_1, B_2, \dots, B_c$ be $m$-bit standard Bloom filters that encode $n_1, n_2, \dots, n_c$ items, respectively, using $k$ hash functions. If $|B_{\land}|$ is the bit sum of $B_{\land}$, then
        \begin{equation} \label{eqn:expected_intersection}
        \mathbb{E}\left[ \, | B_{\land} | \, \right] = \frac{1}{m^{n k -1}} \, \prod\limits_{i=1}^c m^{n_i k} - (m-1)^{n_i k} \qquad \mbox{and}
        \end{equation}
        \begin{align} \label{eqn:var_intersection}
        \mathbb{V}\!\mbox{ar}\left[ \, | B_{\land} | \, \right] = \, & \frac{1}{m^{n k - 1}} \left( (-1)^c \prod\limits_{i=1}^c \left( m^{n_i k} - (m-1)^{n_i k} \right) \right. \\
        \notag & \, + \left. (m - 1) \prod\limits_{i=1}^c \left( m^{n_i \, k} - 2(m-1)^{n_i k} + (m-2)^{n_i k} \right) \right) \\
        \notag & \, + \frac{1}{m^{2(n k -1)}} \left( \, \prod\limits_{i=1}^c \left( m^{n_i k} - (m-1)^{n_i k} \right)^2 \right)
        \end{align}
    where $n = \sum_{i=1}^c n_i$.
\end{corollary}
\begin{proof}
    Equations \eqref{eqn:expected_intersection} and \eqref{eqn:var_intersection} follow directly from Theorem~\ref{thm:mutual-batch-occupy-moment} and Lemma~\ref{cor:classic-mean-and-variance}.
\end{proof}

Corollary~\ref{cor:filter-intersection} and the estimators of Section~\ref{subsec:estimators} can be used to quantify the statistical similarity between a collection of Bloom filters, so caution should be exercised when using Bloom filters to store either personally identifiable information (PII) or sensitive personal information (SPI).


\subsection{False-Positive Rates}
\label{subsec:false-positive}

The false-positive rate for a filter is the probability that a randomly selected $x \in D$ passes the membership check for $S$ but $x \not\in S$. 
Conventionally, it is assumed in false-positive analysis that $|S| \ll |D|$ so a positive membership check is almost surely a false-positive result.
Hence, the false-positive rate reduces to the probability of sampling a positive result.
For a Bloom filter, this is the probability that the $k$ hash bits for $x$ are already marked in the filter for $S$:
    \begin{equation}
    f(B) = \mathbb{P}\left[ b_{h_1(x)} = b_{h_2(x)} = \dots = b_{h_k(x)} =  1\right] \, .
    \end{equation}
The expected false positive rates for standard and classic Bloom filters with identical parameters are different which sections \ref{subsec:standard-false-positive} and \ref{subsec:classic-false-positive} elucidate.


\subsubsection{Standard Bloom Filters}
\label{subsec:standard-false-positive}

The hash function $H:D \rightarrow [m]^k$ for a standard Bloom filter is constructed using $k$ independent hash functions $h_i : D \rightarrow [m]$ for $1 \le i \le k$, so
    \begin{equation}
    \mathbb{P}\left[ b_{h_1(x)} = b_{h_2(x)} = \dots = b_{h_k(x)} =  1\right] = \mathbb{P}\left[ b_{h_1(x)} =  1 \right]^k = \left(\frac{|B_S|}{m}\right)^k
    \end{equation}
where $|B_S|$ is the number of bits set to one in the $m$-bit standard Bloom filter $B_S$. 
The number of 1-bits $|B_S|$ corresponds to the occupancy number $X_S$ in the classic occupancy model, so the expected false positive rate is
    \begin{equation} 
    \label{eq:bloom-false-positive}
    f_S(m,n,k) := \mathbb{E}[f(B_S)]  
        = \sum\limits_{i=1}^m \, \mathbb{P}[X_S = i] \left( \frac{i}{m} \right)^k
        = \frac{\mathbb{E}[X_{S}^{k}]}{m^k}
    \end{equation}
and can be computed from the $k$\textsuperscript{th} ordinary moment of the classic occupancy distribution.
Formula \eqref{eq:classic-moments-1} gives
    \begin{equation}
    \label{eq:standard-fp-in-m}
    f_S(m,n,k) = \frac{1}{m^{(n+1)k}} \sum\limits_{i=1}^m \stirling{nk}{i} \, m^{\underline{i}} \; i^k \, ,
    \end{equation}
which agrees with Bose et al. \cite{Bose2008false}.
Grandi~\cite{Grandi2018Analysis} gives the recursive formula
    \begin{equation}
    \label{eqn:standard-bloom-recurrence}
    \psi_S(h,m) =  \psi_S(h-1,m) - \left(1-\frac{1}{m}\right)^{kn+h+1} \psi_S(h-1,m-1)
    \end{equation}
where $f_S(m,n,k) = \psi_S(k,m)$, which is a corollary of Theorem~\ref{thm:recursive-moments}
    \begin{align*}
    \mathbb{E}\left[\left(\frac{X}{m}\right)^{\! k} \mid m, nk \right] & = \mathbb{E} \left[ \left( \frac{X}{m} \right)^{k+1} \mid m, nk \right] \\
    & \qquad + \left(1 - \frac{1}{m} \right)^{nk + k} \mathbb{E}\left[ \left(\frac{X}{m-1} \right)^k \mid m-1, nk \right] \, .
    \end{align*}

Several authors \cite{Bose2008false,Christensen2010new,Grandi2018Analysis} argue that the sum \eqref{eq:standard-fp-in-m} is intractable for practical parameters.
In many Bloom filter applications, $m$ is large and the sum involves many high precision fractional summands -- which quickly overflows the memory of most modern computers.
However, Lemma~\ref{lem:classic-moments-2} can be used to express the false-positive rate as a sum in terms of $k$, which is generally much smaller than $m$.


\begin{theorem} 
    \label{thm:new-false-positive}
    For $m,n,k \in \mathbb{N}$, the expected false-positive rate of a standard Bloom filter is
        \begin{equation} 
        \label{eqn:standard-bloom-fp}
        f_S(m,n,k) = \frac{1}{m^{(n+1)k}} \sum\limits_{i=0}^k \stirling{k}{i} \, m^{\underline{i}} \, \sum\limits_{j=0}^i (-1)^j \binom{i}{j} (m-j)^{nk} \, .
        \end{equation}
\end{theorem}
\begin{proof}
    From \eqref{eq:bloom-false-positive} and Lemma~\ref{lem:classic-moments-2},
        \begin{equation*}
        \label{eq:fp-occupy-moments}
        f_S(m,n,k) = \frac{\mathbb{E}[X_S^k]}{m^k} = \frac{1}{m^{k}} \sum\limits_{i=0}^k \stirling{k}{i} \, m^{\underline{i}} \, \frac{ \nabla^{i} m^{nk} }{ m^{nk} } \, .
        \end{equation*}
    and $\nabla^i m^{nk} = \left. \nabla^i x^{nk} \right|_{x=m} = \sum_{j=0}^i (-1)^j \binom{i}{j} (m-j)^{nk}$.
\end{proof}

Since the false-positive rate can be expressed a function of expectation -- i.e., $f_S = \mathbb{E}\left[ \varphi(X) \right]$ where $\varphi(X) = \left(\frac{X}{m}\right)^k$ -- its partial Taylor series gives the approximation
    \begin{equation*}
    \label{eq:taylors-series-estimate}
    \mathbb{E}\left[ \varphi(X_S) \right] = \mathbb{E}\left[ \sum\limits_{i=0}^{\infty} \frac{\varphi^{i}(\mu_S)}{i!} \, (X_S-\mu_S)^i\right] 
     \approx \varphi\left(\mu_S\right) + \frac{\sigma^2_S}{2} \, \varphi^{\prime\prime}\left( \mu_S \right)
    \end{equation*}
where $\mu_S = \mathbb{E}[X_S]$, which is a close approximation \cite{Grandi2018Analysis} and can readily be computed using \eqref{cor:classic-mean-and-variance} with $nk$ balls.
Jensen's inequality shows that the the well-known constant approximations
    \begin{equation}
    \label{eqn:standard-fp-lower-bound}
    M(m,n,k) = \left(1-\left( 1 - \frac{1}{m} \right)^{k n} \right)^k
    \end{equation}
\cite{Severance1976Differential,Bose2008false,Broder2003Applications,Christensen2010new}, and 
    \begin{equation}
    \label{eqn:exp-fp-lower-bound}
    E(m,n,k) = \left( 1 - e^{-kn/m} \right)^k
    \end{equation}
\cite{Knuth1973Art} are a lower bounds for $f_S$, since
    \begin{equation*}
    E(m,n,k) \le M(m,n,k) = \varphi(\mu_S) \le \mathbb{E}[\varphi(X_S)] = f(m,n,k) \, .
    \end{equation*}
Bloom's approximation \cite{Bloom1970Space} for the classic filter $\left( 1 - \left( 1 - \frac{k}{m} \right)^n \right)^k$ turns out to be an upper bound for both the classic and standard Bloom filter; see Corollary~\ref{cor:fp-upper-lower-bounds}.


\subsubsection{Classic Bloom Filters}
\label{subsec:classic-false-positive}

    \begin{figure}
        \centering
    	\begin{overpic}[scale=0.75]{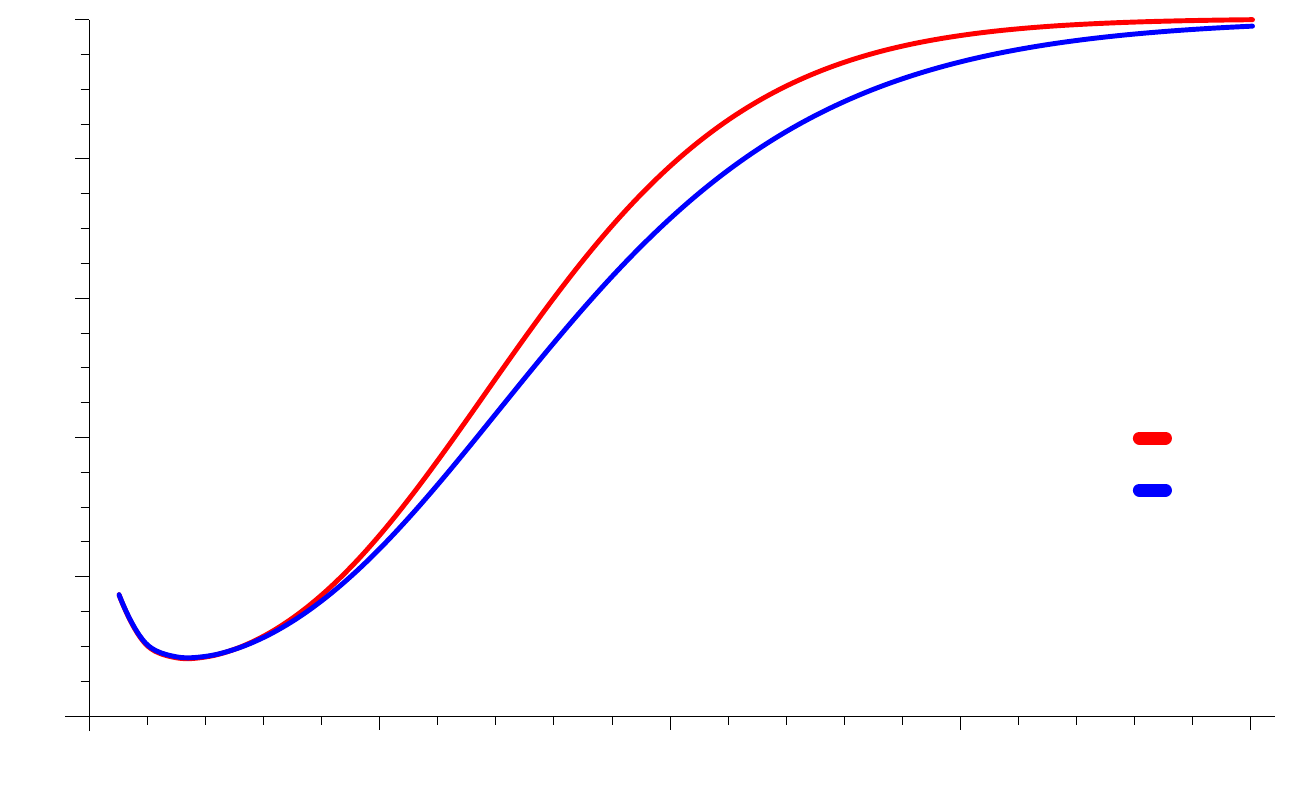}
    		\put(2,166){1.0}
    		\put(2,136){0.8}
        	\put(2,105){0.6}
        	\put(2,75){0.4}
        	\put(2,45){0.2}
        	\put(0,15){0.0}
        	\put(17,5){0}
        	\put(78,5){10}
        	\put(140,5){20}
        	\put(203,5){30}
        	\put(266,5){40}
            \put(280,15){$k$}
    		\put(257,75){$f_C$}
    		\put(257,64){$f_S$}
        \end{overpic}
        \caption{False-positive rates for classic (red) and standard (blue) Bloom filters of length $m=100$ storing $n=20$ items using $k=1,\dots,40$ hash bits per item.}
        \label{fig:fp-vs-k}
    \end{figure}
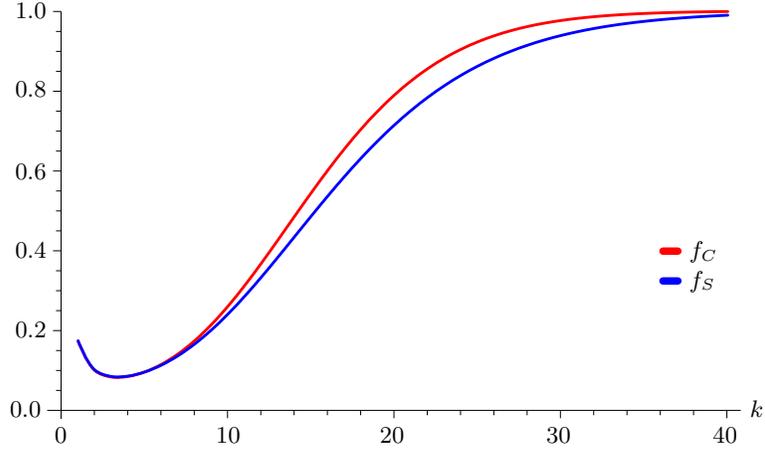

The false positive rate for an $m$-bit classic Bloom filter generated by marking $k$ bits per item is
    \begin{equation*}
    \mathbb{P}\left[ b_{h_1(x)} = b_{h_2(x)} = \dots = b_{h_k(x)} =  1\right] = \frac{\binom{|B_C|}{k}}{\binom{m}{k}} .
    \end{equation*}
The bit sum $|B_C|$ for a classic Bloom filter $B_C$ follows a committee distribution, so the expected false-positive rate is
    \begin{equation}
    \label{eqn:classic-fp-factorial-moments}
    f_C(m,n,k) := \mathbb{E}[f(B_C)]
        = \mathbb{E}\left[ \frac{ \binom{X_C}{k} }{ \binom{m}{k} } \right] 
        = \frac{ \mathbb{E}[\binom{X_C}{k} ] }{ \binom{m}{k} } \, .
    \end{equation}
From \eqref{eqn:batch-factorial-moments} and Theorem~\ref{thm:psi-formulas}, this can be rewritten as
    \begin{align}
    \label{eqn:classic-bloom-fp}
    f_C (m,n,k) & = \sum\limits_{i=0}^{k} (-1)^i \binom{k}{i} \left[ \frac{\binom{m-i}{k}}{\binom{m}{k}} \right]^n \\
    & = \frac{ \nabla^k \left[ \binom{x}{k}^n \right]_{x=m} }{ \binom{m}{k}^n } = \rho_{\vec{k}}(k,m) \\
        & = \pFq{n+1}{n}{-k,k-m,\dots,k-m}{-m,\dots,-m}{1} \, .
    \end{align}
Grandi~\cite{Grandi2018Analysis} gives the recursive definition
    \begin{equation}
    \label{eqn:classic-bloom-recurrence}
    \psi_C(h,m) = \psi_C(h-1,m) - \left(1 - \frac{k}{m} \right)^{n} \psi_C(h-1,m-1)
    \end{equation}
such that $f_C(m,n,k)=\psi_C(m,k)$, which is a consequence of Theorem~\ref{thm:psi-formulas}.

Both the false-positive rate formulas \eqref{eqn:standard-bloom-fp} and \eqref{eqn:classic-bloom-fp} can be difficult to compute, but the following corollary provides some simple bounds.


\begin{corollary}
    \label{cor:fp-upper-lower-bounds}
    For all $m,n \in \mathbb{N}$ and $1\le k \le \frac{m-1}{2}$, 
        \begin{equation*}
        L(m,n,k) \le f_S(m,n,k), f_C(m,n,k) \le U(m,n,k)
        \end{equation*}
    where
        \begin{align}
        \label{eqn:fp-lower-bound}
        L(m,n,k) = \frac{\nabla^k \left[ x^{nk} \right]_{x=m}}{m^{n k}} \ge \frac{\binom{\mu_S}{k}}{\binom{m}{k}} \qquad \mbox{and} \\
        \label{eqn:fp-upper-bound}
        U(m,n,k) = \left( 1 - \left( 1 - \frac{k}{m} \right)^n \right)^k = \left( \frac{\mu_C}{m} \right)^k
        \end{align}
    such that $\mu_S = m \left( 1 - \left( 1 - \frac{1}{m} \right)^{n k} \right)$ and $\mu_C = m \left( 1 - \left( 1 - \frac{k}{m} \right)^n \right)$.
\end{corollary}
\begin{proof}
    Theorem~\ref{thm:batch-occupancy-bounds} shows that the upper and lower bounds hold for the classic Bloom filter.
    
    Fix $\mu = \mu_C = m \left( 1 - \left( 1 - \frac{k}{m} \right)^n \right)$. For the standard Bloom filter, 
        \begin{align*}
        U(m,n,k) & = \frac{\mu^k}{m^k} = \frac{1}{m^k} \sum\limits_{i=0}^k \stirling{k}{i} \mu^{\underline{i}},
        \end{align*}
    which, by Theorem~\ref{thm:batch-occupancy-bounds}, gives
        \begin{align*}
        \frac{ \mu^{\underline{i}} }{ m^{ \underline{i} }} 
            = \frac{ \binom{\mu}{i} }{ \binom{m}{i} } 
            \ge \frac{\nabla^i \! \left[ x^{n k} \right]_{x=m}}{m^{n k}}, 
        \quad \mbox{i.e.,} \quad 
        \mu^{\underline{i}} \ge m^{\underline{i}} \frac{\nabla^i \! \left[ x^{n k} \right]_{x=m}}{m^{n k}} \, .
        \end{align*}
    Hence, $U(m,n,k) \ge \sum_{i=0}^k \stirling{k}{i} \frac{m^{\underline{i}}}{m^k} \frac{\nabla^i [x^{n k}]_{x=m} }{m^{n k}} = f_S(m,n,k)$.
    
    The lower bound follows from Jensen's inequality:
        \begin{align*}
        f_S(m,n,k) = \frac{\mathbb{E}\left[X_S^k\right]}{m^k} 
            \ge \frac{\mathbb{E}\left[X_S \right]^k}{m^k} 
            = \left( 1 - \left( 1 - \frac{1}{m} \right)^{n k} \right)^k \, ,
        \end{align*}
    and -- just as in Theorem~\ref{thm:batch-occupancy-bounds} -- recursion gives
        \begin{align*}
        \frac{\nabla^{i} \, \left[ x^{n k} \right]_{x=m}}{m^{n k}} 
            & = \frac{1}{m^{n k}} \nabla^{i-1} \! \left[x^{n k} - (x-1)^{n k} \right]_{x=m} \\
            & = \frac{1}{m^{n k}} \nabla^{i-1} \! \left[x^{n k} \left( 1  - \left( 1-\frac{1}{x} \right)^{n k} \right) \right]_{x=m} \\
            & \le \left( 1 - \left( 1 - \frac{1}{m} \right)^{nk} \right) \frac{\nabla^{i-1} \left[ x^{n k}\right]_{ x=m }}{ m^{n k} } \, ,
        \end{align*}
    so $f_S(m,n,k) \ge \frac{\nabla^k \left[ x^{n k} \right]_{x=m}}{ m^{n k} } = \frac{\mathbb{E}[X_S^{\underline{k}}]}{m^{\underline{k}}}$.
    The simpler bound $\binom{\mu_S}{k} / \binom{m}{k}$ can be obtained by applying Jensen's inequality to the factorial moments. 
\end{proof}

    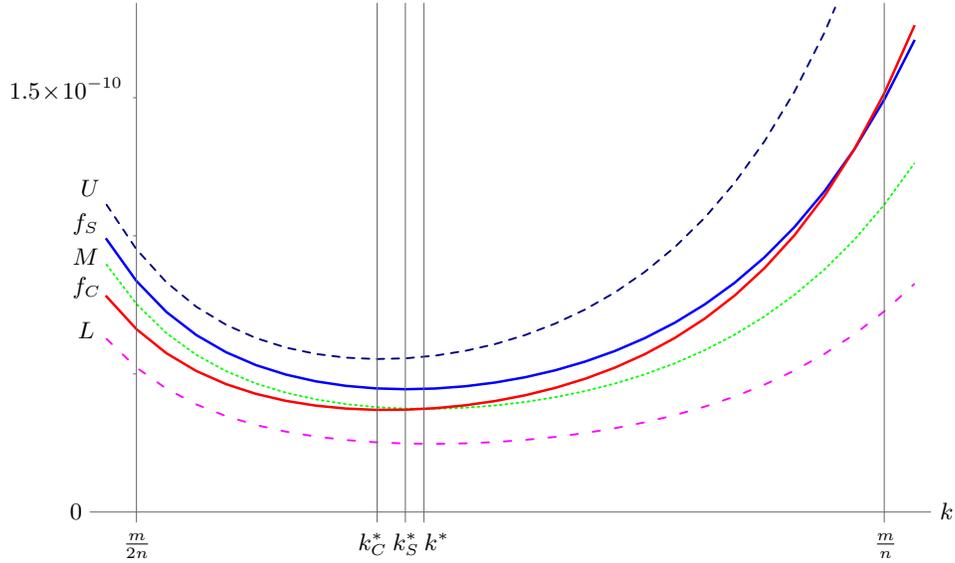
\begin{figure}
        \centering
    	\begin{overpic}[scale=0.33]{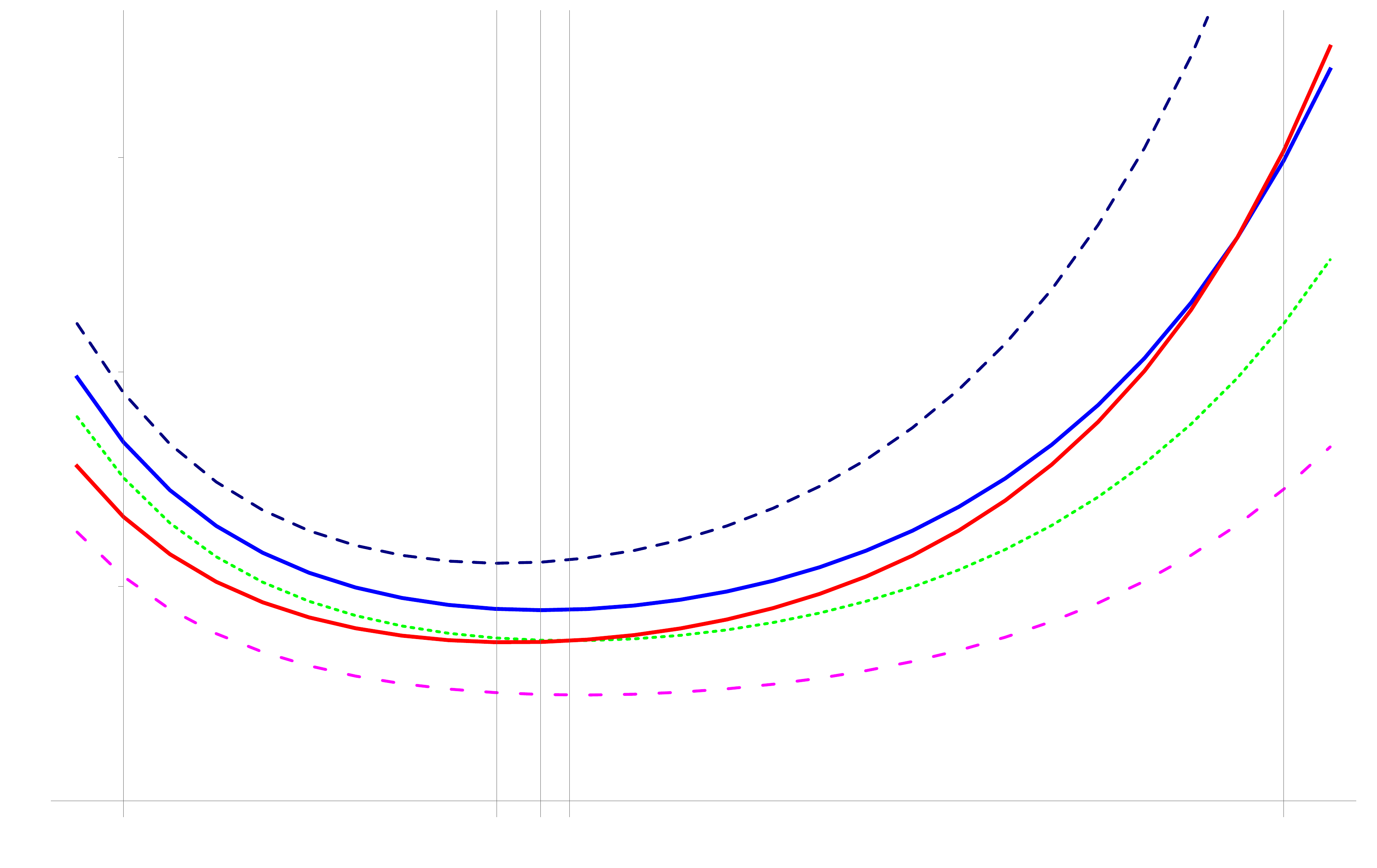}
    		\put(-18,173){$1.5\!\times\!10^{-10}$}
        	\put(5,13.5){0}
        	\put(25,2){$\frac{m}{2n}$}
        	\put(114,2){$k^*_C$}
        	\put(127,2){$k^*_S$}
        	\put(139,2){$k^*$}
        	\put(309,2){$\frac{m}{n}$}
            \put(334,13.5){$k$}
    		\put(9,136){$U$}
    		\put(6,123){$f_S$}
    		\put(6,110){$M$}
    		\put(6,99){$f_C$}
    		\put(8,82){$L$}
        \end{overpic}
        \caption{Schematic of the upper $(U)$, middle $(M)$, and lower $(L)$ bounds for the false-positive rates of classic $(f_C)$ and standard $(f_S)$ Bloom filters using the parameters: $m=1000$, $n=20$, and $k=25,\dots,50$. The optimal number of hash bits $k_C^*=33$ and $k_S^*=34$ and the estimate $k^* = \frac{m}{n} \ln 2 \approx 34.7$ are marked.}
        \label{fig:fp-bounds}
    \end{figure}


\subsection{Historical Notes}
\label{sec:historical-notes}

Bloom filters are an application of randomized superimposed codes~\cite{Roberts1979Partial}, which played a central role in the Zatocoding retrieval systems \cite{Mooers1947Zatocoding} developed during the 1940s and 50s.
In 1970, Bloom \cite{Bloom1970Space} described the classic Bloom filter algorithm for obtaining an approximate answer to the set membership problem using a fixed-length superimposed code, and estimated the false-positive rate based on the expected bit occupancy as $U(m,n,k)$ in \eqref{eqn:fp-upper-bound}.
The novelty of Bloom's method lies in the allowance of errors to increase efficiency in time and space.
Roberts~\cite{Roberts1979Partial} reported the exact formula \eqref{eqn:classic-bloom-fp} for the classic Bloom filter in 1979.

The standard Bloom filter and false-positive estimate $E(m,n,k)$ in \eqref{eqn:exp-fp-lower-bound} likely originated in \cite{Knuth1973Art} as misreporting of Bloom's classic construction \cite{Bloom1970Space}.
Knuth remarks that the section detailing Bloom's method was written in haste during the April of 1972 \cite{Knuth1972Bloom,Knuth1973Art}.
Possible reasons for the continued misattribution of the standard Bloom filter could include:
1) the similarities between the filter constructions, 
2) the fact that the false-positive rate approximations and asymptotics coincide for large filters, and 
3) \textit{The Art of Computer Programming} \cite{Knuth1973Art} is a seminal work and often used as a primary reference.
In 1976, Severance and Lohman \cite{Severance1976Differential} provide the false-positive approximation $M(m,n,k)$ in \eqref{eqn:standard-fp-lower-bound} and list \cite{Bloom1970Space} and \cite{Knuth1973Art} as references but only cite Bloom.

Bose et al.~\cite{Bose2008false} derived the exact formula \eqref{eq:standard-fp-in-m} for the standard Bloom filter in 2008, and credit the theoretical and empirical discrepancies observed by \cite{Gremillion1982Designing,Mullin1983Second} to a mistake in Bloom's analysis while failing to note the difference between the classic and standard filter constructions.
Bloom's original exposition \cite{Bloom1970Space} is nearly correct, but he computes the false-positive rate of the expected filter $f(\mathbb{E}[B])$ rather than the expected false-positive rate of a filter $\mathbb{E}[f(B)]$.

Grandi~\cite{Grandi2018Analysis} provided the first side-by-side analysis of the false-positive rates for classic and standard Bloom filters in 2018, almost 50 years after their inventions.
As fate would have it, Sections~\ref{sec:filter-optimization} and \ref{sec:filter-efficiency} show that Bloom's original construction is more efficient than the better known filter bearring his name.


\subsection{False-Positive Rate Minimization}
\label{sec:filter-optimization}

The false-positive rate of a (classic or standard) Bloom filter can be decreased by
\begin{enumerate}
    \item increasing the filter size $m$,
    \item decreasing the number of stored items $n$, or
    \item optimizing the number of hash bits $k$ used to encode each item.
\end{enumerate}
For a fixed $m=\dot{m}$ and $n=\dot{n}$, the false positive rate using an optimal number of hash bits $k^*$ satisfies 
    \begin{equation*}
    f^*(\dot{m},\dot{n}) := f(\dot{m},\dot{n},k^*) \le f(\dot{m},\dot{n},k) \, \, \mbox{ for all } k \in \mathbb{N} \, .
    \end{equation*}
If the false-positive rate $p=\dot{p}$ and $m=\dot{m}$ are fixed, then the maximum $n_{max}$ can be found such that
    \begin{equation*}
    f^*(\dot{m},n) > \dot{p} \, \, \mbox{ for all } n \ge n_{max} \, .
    \end{equation*}
Likewise, if $p=\dot{p}$ and $n=\dot{n}$ are fixed, the minimum $m_{min}$ can be computed such that
    \begin{equation*}
    f^*(m,\dot{n}) > \dot{p} \, \, \mbox{ for all } m \le m_{min} \, .
    \end{equation*}
Calculating $k^*$, $n_{max}$, or $m_{min}$ involves computing the filter false-positive rate, and the approximation $E(m,n,k)$ from \eqref{eqn:exp-fp-lower-bound} is often used in optimization derivations, since
    \begin{equation}
    \label{eqn:fp_optimal_approx}
    f(m,n,k) \approx f_S(m,n,k) \approx (1-e^{-nk/m})^k
    \end{equation}
and the exact formulas \eqref{eqn:standard-bloom-fp} and \eqref{eqn:classic-bloom-fp} for standard and classical Bloom filters are complicated. 
In particular, for fixed $n=\dot{n}$ and $m=\dot{m}$:
    \begin{equation}
    \label{eqn:approx_optimal_k}
    k^* = \min_{k\in[\dot{m}]} f(\dot{m},\dot{n},k) \approx \frac{\dot{m}}{\dot{n}} \ln 2
    \end{equation}
which can be verified by checking that $\frac{d}{dk}\left[ (1-e^{-nk/m})^k \right] = 0$ at $k=\frac{m}{n} \ln 2$.
Noting that the false-positive rate is minimized at $k^*$, approximations for $n_{max}$ and $m_{min}$ for a given false-positive rate $p=\dot{p}$ can be obtained through
    \begin{equation*}
    \hat{p} = f(\dot{m},\dot{n},k^*) \approx (1-e^{\dot{n} \, {k^*} / \dot{m}})^{k^*} = 2^{ - \frac{\dot{m}}{\dot{n}} \ln 2}
    \end{equation*}
which implies that $\frac{\dot{m}}{\dot{n}} \approx -\frac{\log_2 \dot{p}}{\ln 2}$, so
    \begin{equation}
    \label{eqn:approx_optimal_m_and_n}
     \quad m_{min} \approx - \frac{ \dot{n} \log_2 \dot{p}}{ \ln 2 } \quad \mbox{and} \quad n_{max} \approx - \frac{ \dot{m} \ln 2 }{ \log_2 \dot{p}} \, .
    \end{equation}
Approximations \eqref{eqn:approx_optimal_k} and \eqref{eqn:approx_optimal_m_and_n} hold when \eqref{eqn:fp_optimal_approx} holds and differ when $n \ll m$, as shown by Fig~\ref{fig:optimal-k}.

For a concrete example, when $m=64$ and $n=4$ the approximations \eqref{eqn:fp_optimal_approx} and \eqref{eqn:approx_optimal_k} give $k^* \approx 11.09$ and $f^*(64,4) \approx 4.59 \times 10^{-4}$ for both the classic and standard Bloom filters.
However, $f^*_S(64,4) \approx 6.15 \times 10^{-4}$ when $k^*_S = 10$,  and $f^*_C(64,4) \approx 4.55 \times 10^{-4}$ when $k^*_C = 9$.
The estimate ${k^*}=11$ leads to false-positive rates of $f_S(64,4,11) \approx 6.25 \times 10^{-4}$ and $f_C(64,4,11) \approx 4.85 \times 10^{-4}$ which are not optimal.

    \begin{figure}
    	\begin{overpic}[scale=0.5]{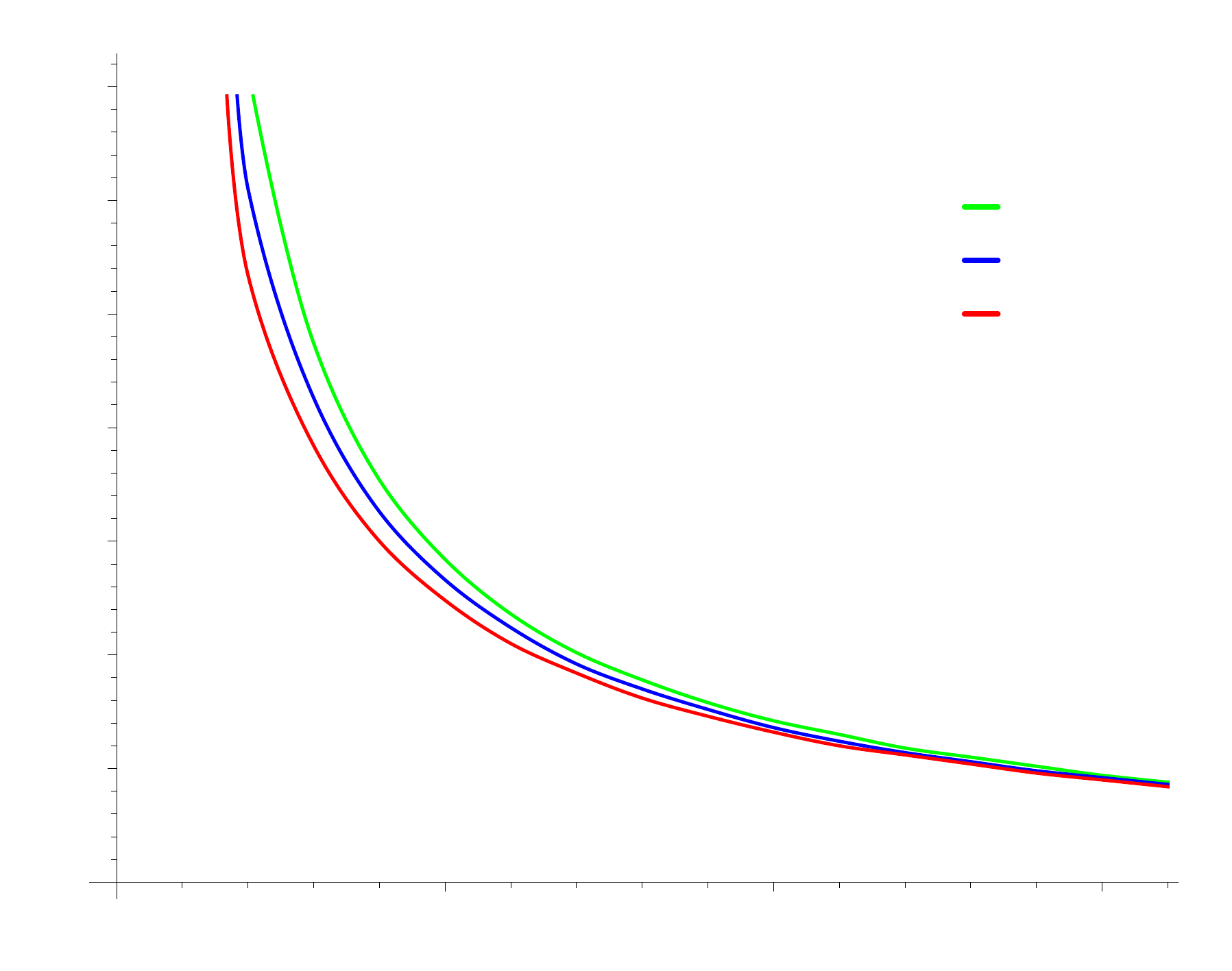}
    		\put(3,183){350}
        	\put(3,159){300}
        	\put(3,136){250}
        	\put(3,112){200}
        	\put(3,87){150}
        	\put(3,65){100}
        	\put(8,40){50}
            \put(22,200){$k$}
        	\put(22,4){0}
        	\put(92,4){5}
        	\put(158,4){10}
        	\put(228,4){15}
            \put(255,16){$n$}
            \put(213,159){$k^*$}
            \put(213,147){$k^*_S$}
            \put(213,135){$k^*_C$}
        \end{overpic}
        \caption{Optimal hashes bits $k^*$, $k^*_S$, and $k^*_C$ corresponding to the minimum false-positive rates $f^*(\hat{m},n)$, $f^*_S(\hat{m},n)$, and $f^*_C(\hat{m},n)$ for $\hat{m}=1024$.}
        \label{fig:optimal-k}
    \end{figure}
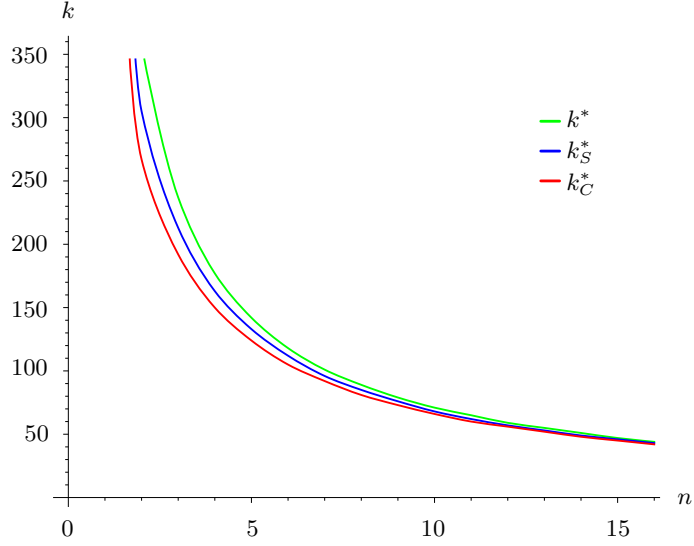

The estimate $k^*$ can yield a non-optimal encoding for the standard and classic Bloom filter; Fig~\ref{fig:fp-bounds} and Fig~\ref{fig:optimal-k} confirm this point pictorially.
The general conclusion inferred from the preceding example is summarized in the following conjecture.


\begin{conjecture}
\label{con:optimal-ks}
For $\frac{m}{n} \ge \frac{1}{\ln 2}$,
    \begin{equation*}
    \frac{m}{2n} \le k_C^* \le k_S^* \le \ln 2 \, \frac{m}{n} \, ,
    \end{equation*}
and $k_C^* = k_S^* = 1$ when $\frac{m}{n} < \frac{1}{ \ln 2 }$.
\end{conjecture}

For fixed $m$ and $n$, choosing $k$ to minimizes the false-positive rate is most efficient use of the filter.
Deciding whether a decreased false-positive justifies an increase in $m$ or a decrease is $n$ is harder to determine. 
This notion filter efficiency is codified in the following section.

%% file: fig/bloom-false-positive.tex
\begin{tikzpicture}

	\draw[step=0.75,gray,thin] (0,0) grid (8.25,0.75);

	\node at (0.375,0.375) {0};
	\node at (1.125,0.375) {0};
	\node at (1.875,0.375) {1};
	\node at (2.625,0.375) {1};
	\node at (3.375,0.375) {0};
	\node at (4.125,0.375) {1};
	\node at (4.875,0.375) {0};
	\node at (5.625,0.375) {0};
	\node at (6.375,0.375) {1};
	\node at (7.125,0.375) {1};
	\node at (7.875,0.375) {0};

	\node at (2.0,2.75) {$x_1$};
	\draw [decoration={brace}, decorate] (0.25,2.375) -- (3.75,2.375);
    
	\node at (0.75,2.0) {$h_1(x_1)$};
    \draw[>=stealth,->,thin] (0.75,1.75) .. controls (0.875,1.25) and (2.375,1.5) .. (2.55,0.85);

	\node at (2.0,2.0)  {$h_2(x_1)$};
    \draw[>=stealth,->,thin] (2.00,1.75) .. controls (1.95,1.5) and (1.9,1.25) .. (1.875,0.85);

	\node at (3.25,2.0) {$h_3(x_1)$};
    \draw[>=stealth,->,thin] (3.25,1.75) .. controls (3.5,1.25) and (6.25,1.5) .. (6.375,0.85);

	\node at (6.25,2.75) {$x_2$};
	\draw [decoration={brace}, decorate] (4.5,2.375) -- (8.0,2.375);

	\node at (5.0,2.0)  {$h_1(x_2)$};
    \draw[>=stealth,->,thin] (5.0,1.75) .. controls (4.5,1.25) and (2.875,1.5) .. (2.7,0.85);

	\node at (6.25,2.0) {$h_2(x_2)$};
    \draw[>=stealth,->,thin] (6.25,1.75) .. controls (6.375,1.25) and (7.0,1.5) .. (7.125,0.85);

	\node at (7.5,2.0)  {$h_3(x_2)$};
    \draw[>=stealth,->,thin] (7.5,1.75) .. controls (7.375,1.25) and (4.375,1.5) .. (4.125,0.85);

	\node at (4.125,-2.25) {$x_3$};
	\draw [decoration={brace}, decorate] (5.875,-1.875) -- (2.375,-1.875);

	\node at (2.875,-1.5) {$h_1(x_3)$};
    \draw[>=stealth,->,thin,dotted] (2.875,-1.2) .. controls (3.0,-0.75) and (4.0,-0.6) .. (4.125,-0.1);

	\node at (4.125,-1.5) {$h_2(x_3)$};
    \draw[>=stealth,->,thin,dotted] (4.125,-1.2) .. controls (4.0,-0.75) and (2.75,-0.6) .. (2.625,-0.1);

	\node at (5.375,-1.5) {$h_3(x_3)$};
    \draw[>=stealth,->,thin,dotted] (5.375,-1.2) .. controls (5.5,-0.75) and (6.25,-0.6) .. (6.375,-0.1);

\end{tikzpicture}

%% file: fig/bloom_union_intersection.tex
\begin{tikzpicture}

	\node at (-0.4,2.579) {$B_1$};
	\draw[step=0.4,gray,thin] (0,2.399) grid (4.0,2.8);
	\node at (0.2,2.6) {1};
	\node at (0.6,2.6) {0};
	\node at (1.0,2.6) {1};
	\node at (1.4,2.6) {0};
	\node at (1.8,2.6) {0};
	\node at (2.2,2.6) {1};
	\node at (2.6,2.6) {0};
	\node at (3.0,2.6) {1};
	\node at (3.4,2.6) {0};
	\node at (3.8,2.6) {1};

	\node at (-0.4,1.8) {$B_2$};
	\draw[step=0.4,gray,thin] (0,1.599) grid (4.0,2.0);
	\node at (0.2,1.8) {0};
	\node at (0.6,1.8) {0};
	\node at (1.0,1.8) {1};
	\node at (1.4,1.8) {1};
	\node at (1.8,1.8) {0};
	\node at (2.2,1.8) {1};
	\node at (2.6,1.8) {0};
	\node at (3.0,1.8) {0};
	\node at (3.4,1.8) {1};
	\node at (3.8,1.8) {1};

	\node at (-0.4,0.979) {$B_3$};
	\draw[step=0.4,gray,thin] (0,0.799) grid (4.0,1.2);
	\node at (0.2,1.0) {0};
	\node at (0.6,1.0) {0};
	\node at (1.0,1.0) {1};
	\node at (1.4,1.0) {0};
	\node at (1.8,1.0) {0};
	\node at (2.2,1.0) {1};
	\node at (2.6,1.0) {0};
	\node at (3.0,1.0) {0};
	\node at (3.4,1.0) {1};
	\node at (3.8,1.0) {1};

	\node at (-0.4,0.18) {$B_4$};
	\draw[step=0.4,gray,thin] (0,0) grid (4.0,0.4);
	\node at (0.2,0.2) {1};
	\node at (0.6,0.2) {0};
	\node at (1.0,0.2) {1};
	\node at (1.4,0.2) {1};
	\node at (1.8,0.2) {0};
	\node at (2.2,0.2) {1};
	\node at (2.6,0.2) {0};
	\node at (3.0,0.2) {0};
	\node at (3.4,0.2) {1};
	\node at (3.8,0.2) {1};

	\node at (10.05,2.18) {$B_{\lor}$};
	\draw[step=0.4,gray,thin] (5.599,1.999) grid (9.6,2.4);
	\node at (5.8,2.2) {1};
	\node at (6.2,2.2) {0};
	\node at (6.6,2.2) {1};
	\node at (7.0,2.2) {1};
	\node at (7.4,2.2) {0};
	\node at (7.8,2.2) {1};
	\node at (8.2,2.2) {0};
	\node at (8.6,2.2) {1};
	\node at (9.0,2.2) {1};
	\node at (9.4,2.2) {1};

	\node at (10.05,0.58) {$B_{\land}$};
	\draw[step=0.4,gray,thin] (5.599,0.399) grid (9.6,0.8);
	\node at (5.8,0.6) {0};
	\node at (6.2,0.6) {0};
	\node at (6.6,0.6) {1};
	\node at (7.0,0.6) {0};
	\node at (7.4,0.6) {0};
	\node at (7.8,0.6) {1};
	\node at (8.2,0.6) {0};
	\node at (8.6,0.6) {0};
	\node at (9.0,0.6) {0};
	\node at (9.4,0.6) {1};

\end{tikzpicture}

%% file: 04_efficiency.tex

\section{Filter Efficiency}
\label{sec:filter-efficiency}

A Bloom filter is one of many solutions to the approximate set membership problem. 
In addition to the dozens of Bloom filter variants \cite{Luo2018Optimizing}, there are set membership filters constructed from completely different paradigms (e.g., \cite{Fan2014Cuckoo,Pandey2017Counting,Porat2009Optimal,Weaver2014SatFilter,Weaver2018XorSat}).
While some of these alternatives offer functionality beyond membership checking -- such as the ability to count, remove, or list the items -- all filter designs are assessed on their space and time performance.
Filter efficiency is a metric for how well a filter makes use of its available storage.
Formally, the \textit{efficiency} of a filter is given by
    \begin{equation*}
    \varepsilon = \frac{n}{m} \log_2 \frac{1}{p}
    \end{equation*}
where $n$ is the number of items stored, $m$ is the filter size in bits, and $p$ is the false-positive rate.
Walker \cite{Walker2007Filters,Weaver2014SatFilter} showed that filter efficiency is bounded by $0 < \varepsilon \le 1$ when items are sampled from a continuous uniform distribution, i.e., the universe is $\mathcal{D}=[0,1]$ and each hash function uniformly maps $h:\mathcal{D} \to [m]$.
A filter that reaches the efficiency limit can store $n$ items at a false-positive rate of $p$ using $n \log_2 \left( 1 / p \right)$ bits.
The value $\log_2 p^{-1}$ measures the \textit{bits of cut-down}, i.e., expresses the average number of bits required to store each item at a given false-positive rate.


\subsection{Bloom Filter Efficiency}

The false-positive rate of a Bloom filter is a function of filter size, number of stored items, and the number of hash bits, i.e., $p = f(m,n,k)$ and the efficiency can be computed as
    \begin{equation*}
    \varepsilon(m,n,k) = - \frac{n}{m} \log_2 f(m,n,k) \, .
    \end{equation*}
From equations \eqref{eq:bloom-false-positive} and \eqref{eqn:classic-fp-factorial-moments}, Bloom filter efficiency can be directly computed from the factorial and binomial cumulants of the committee distributions $X_S$ and $X_C$, respectively.


\begin{corollary}
    An $m$-bit standard Bloom filter storing $n$ items using $k$ hash functions has efficiency
        \begin{align}
        \varepsilon_S (m,n,k) = \frac{n k}{m} \log_2 m - \frac{n}{m} \log_2 \mathbb{E}\left[ {X_S}^k \right]
        \end{align}
    where $X_S$ is a classic occupancy random variable, and a classic Bloom filter has efficiency
        \begin{align}
        \varepsilon_C (m,n,k) = \frac{n}{m} \log_2 \binom{m}{k} - \frac{n}{m} \log_2 \mathbb{E}\left[ \binom{X_C}{k} \right] \, .
        \end{align}
    where $X_C$ is a committee random variable.
\end{corollary}

The \textit{peak} \textit{efficiency} $\varepsilon^*(m,k)$ of an $m$-bit Bloom filter using $k$ hash bits occurs at $n^*$ such that $\varepsilon(m,n^*,k) \ge \varepsilon(m,n,k)$ for all $n \ge 1$, and the \textit{maximum efficiency} for a Bloom filter is 
    \begin{equation*}
     \varepsilon^*(m) = \max\limits_{n,k} \, \varepsilon(m,n,k) \, .
    \end{equation*}

The formulas for the false-positive rate of the standard Bloom filter \eqref{eq:bloom-false-positive} and classic Bloom filter\eqref{eqn:classic-bloom-fp} are somewhat complicated, so it is tempting to compute the efficiency using an approximation for the false positive rate.
Fig.~\ref{fig:fp-bounds} confirms that
    \begin{equation*}
    f(m,n,k) \approx M(m,n,k) = \left( 1 - \left( 1 - \frac{1}{m} \right)^{n k} \right)^k
    \end{equation*}
is a good estimate for both the standard and classic Bloom filter, and yields constant peak efficiencies
    \begin{equation*}
    \varepsilon_M^*(m,k) = \frac{\ln 2}{m \, \ln \left( \frac{m}{m-1} \right)} \quad \mbox{at} \quad n_M^* = \frac{ \ln 2}{k \ln \left( \frac{m}{m-1} \right)}
    \end{equation*}
for each $1 \le k \le m$.
Similarly, the most common false-positive approximation 
    \begin{equation*}
    f(m,n,k) \approx E(m,n,k) = (1-e^{-nk/m})^k
    \end{equation*}
leads to a similar characterization: constant peak efficiencies of 
    \begin{equation*}
    \varepsilon_E^*(m,k) = \ln 2    \quad   \mbox{at}   \quad   n_E^* = \frac{m}{k} \ln 2 \, .
    \end{equation*}
Noting that $\ln \! \left( \frac{m}{m-1} \right) \approx m$, these two estimates agree that the maximum efficiency of a Bloom filter is $\varepsilon^*(m) \approx \ln 2$.
However, Sections \ref{sec:standard-bloom-efficiency} and \ref{sec:classic-bloom-efficiency} reveal that peak efficiencies of Bloom filters are non-constant and change monotonically according to the number of hash bits used to encode each item.
Using more hash bits increases the peak efficiencies of a classic Bloom filter, but decreases the peak efficiencies of a standard Bloom filter.
While the maximum efficiency of a standard Bloom filter is indeed bounded by $\ln 2$, the maximum efficiency of a classic Bloom filter approaches the information-theoretic limit of 1.


\subsubsection{Standard Bloom Filter Efficiency}
\label{sec:standard-bloom-efficiency}


\begin{theorem}
    \label{thm:decreasing-efficiency}
    Let $\varepsilon_S (m,n,k)$ be the efficiency of an $m$-bit standard Bloom filter that encodes $n$ items using $k$ hash bits. For $n,m \ge 2$, and $k \in \mathbb{N}$,
        \begin{equation}
        \varepsilon_S \left(m,\frac{n}{k},k\right) > \varepsilon_S \left(m,\frac{n}{k+1},k+1\right) \, .
        \end{equation}
\end{theorem}
\begin{proof}
    Fix $n,m,k \in \mathbb{N}$. Observe that
        \begin{equation}
        \label{eq:false-positive-power}
        f_S\left(m,\frac{n}{k+1},k+1\right)^{\frac{1}{k+1}} = \mathbb{E}\left[\left(\frac{X_S}{m}\right)^{k+1}\right]^{\frac{1}{k+1}}
        = \frac{1}{m} \, \mathbb{E} \left[ \left| X_S \right|^{k+1} \right]^{\frac{1}{k+1}} \, .
        \end{equation}
    Let $Z$ be the occupancy number for a classic urn r.v. that models an $m$-bit standard Bloom filter storing $\frac{n}{k}$ items using $k$ hashes, then $X_S \sim Z$ and H\"{o}lder's inequality gives
        \begin{equation} \label{eq:holders-inequality}
        \frac{1}{m} \, \mathbb{E} \left[ \left| X_S \right|^{k+1} \right]^{\frac{1}{k+1}} 
        \ge  \frac{1}{m} \, \mathbb{E} \left[ \left| X_S \right|^k \right]^{\frac{1}{k}}
        = \mathbb{E} \left[ \left( \frac{Z}{m} \right)^k \right]^{\frac{1}{k}} \!
        = f_S \! \left( m, \frac{n}{k}, k \right)^{\frac{1}{k}} \! .
        \end{equation}
    Equality holds when $\mathbb{P} \left[ X_S = i \mid m, n, k \right] = \lambda$ for $i\in [m]$ some constant $\lambda \ge 0$, e.g., when $n=1$.
    
    Let $\varepsilon_S^k = \varepsilon_S(m,n/k,k)$ for a fixed $m,n\ge2$. By \eqref{eq:false-positive-power} and \eqref{eq:holders-inequality},
    	\begin{equation*}
    	2^{\varepsilon_S^k - \varepsilon_S^{k+1}} = \left( \frac{f_S\left(m,\frac{n}{k+1},k+1\right)^{\frac{1}{k+1}}}{f_S\left(m,\frac{n}{k},k\right)^{\frac{1}{k}}} \right)^{\frac{n}{m}} > 1
    	\end{equation*}
    since $\frac{n}{m}>0$. 
    The exponential function $2^x$ strictly increases and $2^{\varepsilon_k} > 2^{\varepsilon_{k+1}}$, therefore $\varepsilon_S^k > \varepsilon_S^{k+1}$.
\end{proof}

    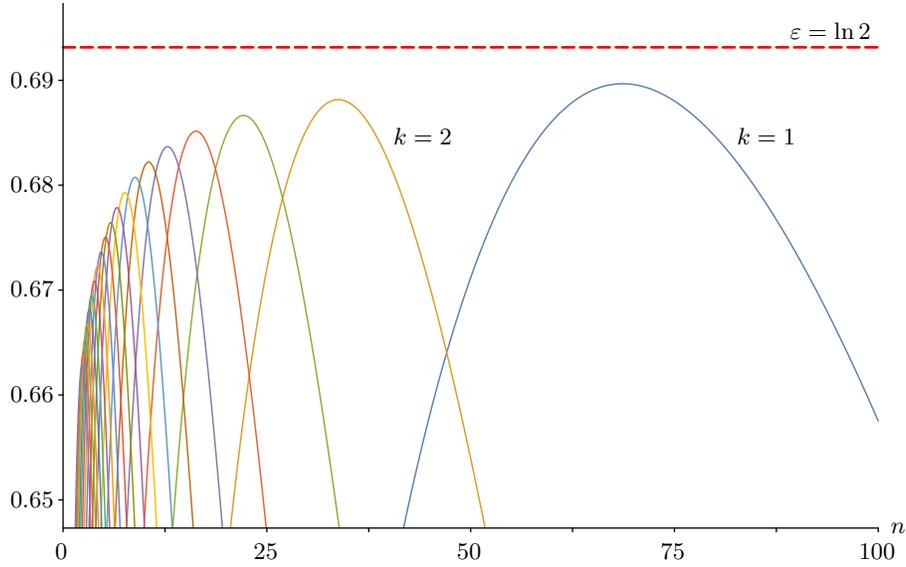
\begin{figure}
    	\begin{overpic}[scale=0.35]{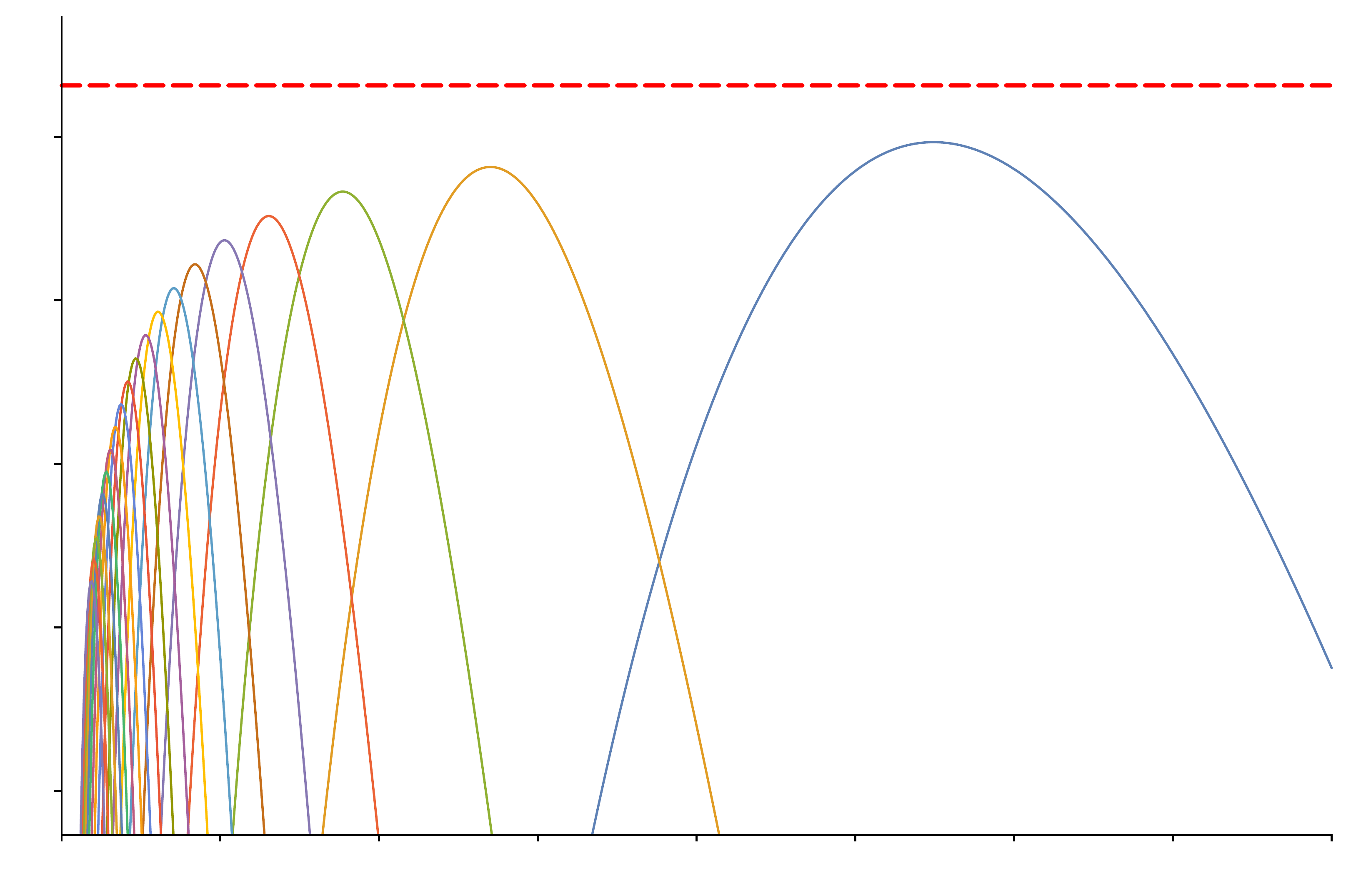}
    		\put(-5,181.5){0.69}
        	\put(-5,142){0.68}
        	\put(-5,102.5){0.67}
        	\put(-5,62.5){0.66}
        	\put(-5,22.5){0.65}
        	\put(12,3){0}
        	\put(87,3){25}
        	\put(164,3){50}
        	\put(241,3){75}
        	\put(316,3){100}
            \put(328,13){$n$}
            \put(290,200){$\varepsilon = \ln 2$}
            \put(270,160){$k = 1$}
            \put(140,160){$k = 2$}
        \end{overpic}
        \caption{Efficiency of an $m=100$ bit standard Bloom filter storing $n$ items encoded using $k=1,2,\dots,20$ hashes. Maximum efficiency $\varepsilon_S^*(100)=0.69$ occurs at $n=69$ and $k=1$.}
        \label{fig:standard-efficiency}
    \end{figure}


\begin{corollary}
    \label{cor:standard-max-efficiency}
    The maximum efficiency of an $m$-bit standard Bloom filter is
        \begin{equation}
        \varepsilon_S^* = \frac{1}{m \log_2\left( \frac{m}{m-1} \right)} \, ,
        \end{equation}
    and is uniquely obtained when $n = \left( \log_2 \frac{m}{m-1} \right)^{-1}$ items are encoded using a single hash function, i.e., $k=1$, at a false positive rate of $p=1/2$.
\end{corollary}
\begin{proof}
    Fix $m \ge 2$. 
    Assume that an $m$-bit Bloom filter achieves maximum efficiency using $k_S^* \ge 2$ hashes for some amount of items $n$. 
    From Lemma~\ref{thm:decreasing-efficiency}, $\varepsilon_S \left( m, \frac{ n }{ k_S^*-1 }, k_S^* - 1 \right) > \varepsilon_S \left( m, \frac{n}{k_S^*}, k_S^* \right)$ and contradicts the assumption that maximal efficiency can be achieved by $k>1$.
    
    The efficiency for Bloom filters that use $k=1$ hash function is
        \begin{equation*}
        \varepsilon_S(m,n,1) = \frac{\ln\left(1-\left(1-\frac{1}{m}\right)^n\right)^{-n}}{m \ln 2} \, .
        \end{equation*}
    Since the natural logarithm strictly increases, the efficiency is maximized when $g(n)=\left( 1 - ( 1 - \frac{1}{m} )^n  \right)^{-n}$ is at a maximum. 
    Setting $x=1-\frac{1}{m}$, gives
        \begin{equation*}
        g'(n)= \frac{(x^n \ln x^n) - ((1-x^n) \ln (1-x^n))}{(1-x^n)^{n+1}}
        \end{equation*}
    which equates to zero when $x^n = 1/2$, i.e., $n = - \ln 2 / \ln \left(1-\frac{1}{m}\right)$.
\end{proof}


\subsubsection{Classic Bloom Filter Efficiency}
\label{sec:classic-bloom-efficiency}

The most efficient standard Bloom filter uses a single bit to encode each item, i.e., the case where the constructions for the standard and classic Bloom filter agree.
Hence the most efficient classic Bloom filter is at least as efficient as a standard Bloom filter.
However, hashing with a single bit realizes the smallest peak efficiency for a classic Bloom filter.


\begin{conjecture}
    \label{con:classic-max-efficiency}
    Let $\varepsilon_C(m,n,k)$ be the efficiency of an $m$-bit classic Bloom filter that encodes $n$ items using $k$ hash bits. For $1 \le k \le \frac{m}{2}-1$,
        \begin{equation*}
        \varepsilon_C^*(m,k) < \varepsilon_C^*(m,k+1)
        \end{equation*}
    and the maximum efficiency of an $m$-bit classic Bloom filter is
        \begin{equation}
        \varepsilon_C^*(m) = \frac{1}{m} \log_2 \binom{m}{m/2} \, ,
        \end{equation}
    and is uniquely obtained when $n=1$ items are encoded using $k=\frac{m}{2}$ hash bits at a false positive rate of $p=\binom{m}{m/2}^{-1}$.
\end{conjecture}

Figure~\ref{fig:classic-efficiency} empirically confirms that the peak efficiencies for the classic Bloom filter increases monotonically with the number of hash bits. 
The case $k=1$ coincides with the standard Bloom filter, i.e., $\varepsilon_C^{*}(m,1)=\varepsilon_S^{*}(m,1) < \ln 2$.
Corollary~\ref{cor:asymp-eff} proves that the peak efficiency for the extremal case $k=m/2$ increases to the efficiency limit of 1 as $m$ increases.


\begin{remark}
    Empirically, the peak efficiency for a classic Bloom filter with fixed $m$ and $k$ is obtained when storing
        \begin{equation}
        n_C^{*} \approx \left(\frac{m}{k} - 1 \right) \, \ln 2
        \end{equation}
    items in the filter. Similarly,
        \begin{equation}
        k_C^{*} \approx \frac{m \, \ln 2}{n + \ln 2}
        \end{equation}
    hash bits is close to the peak efficiency when $m$ and $n$ are fixed.
\end{remark}

    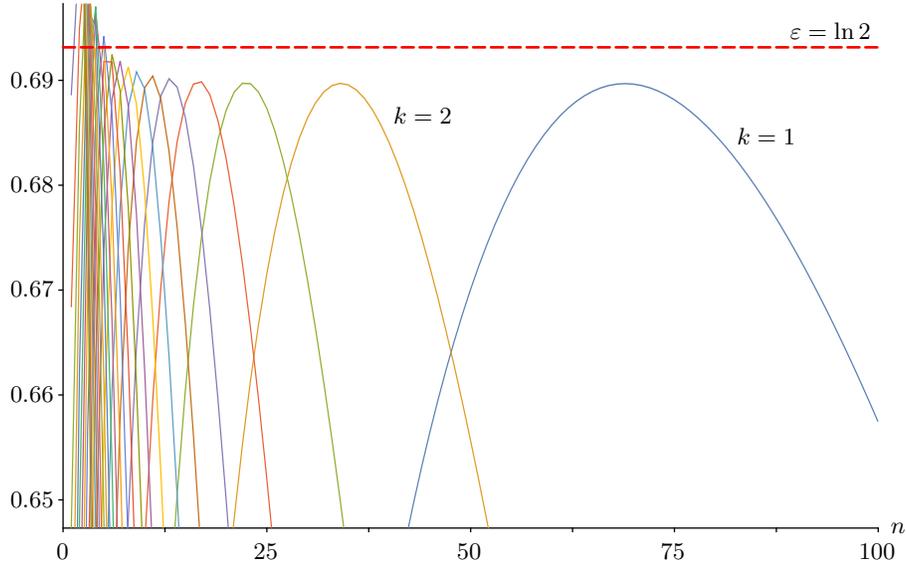
\begin{figure}
    	\begin{overpic}[scale=0.35]{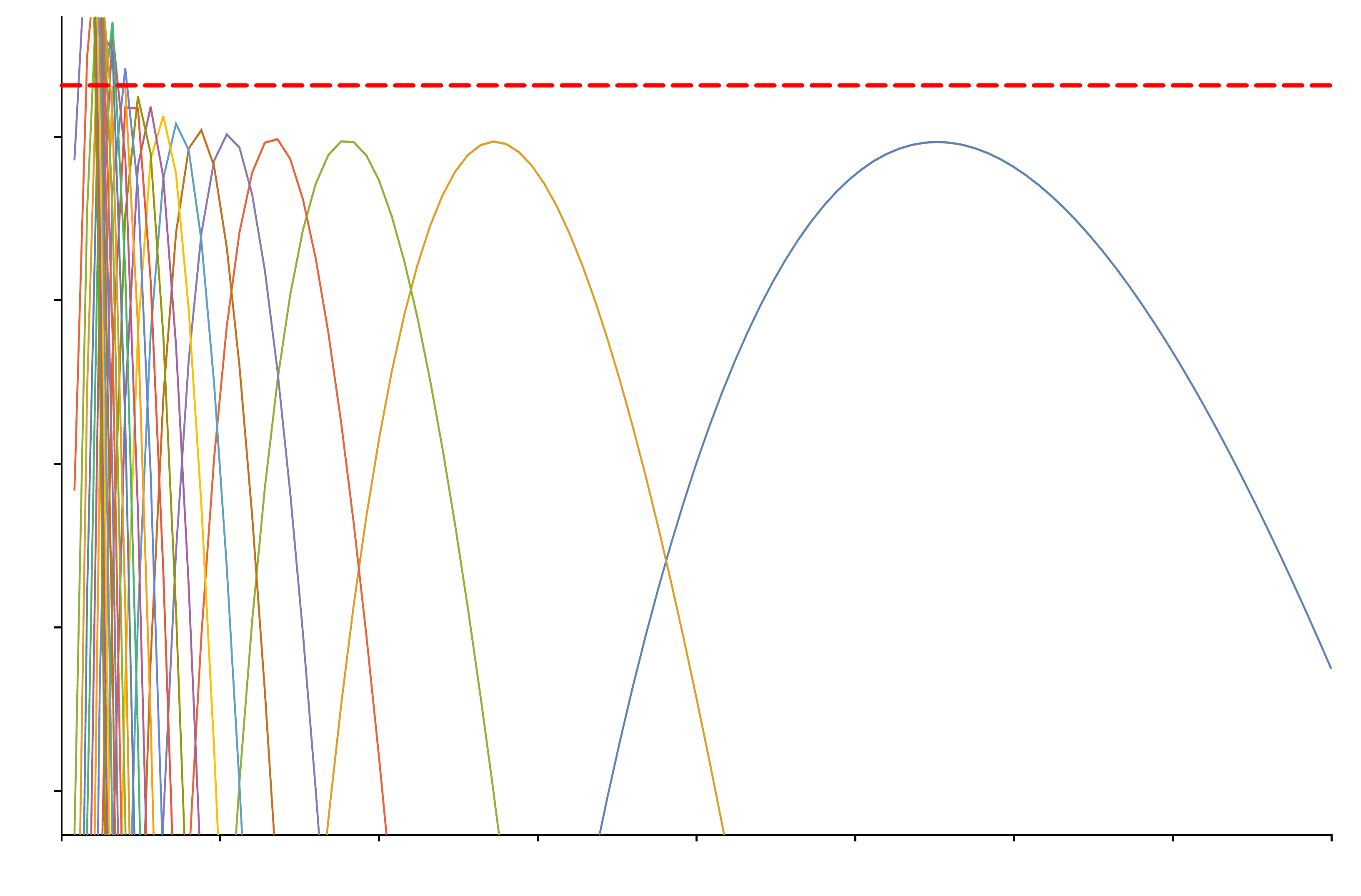}
		    \put(-5,181.5){0.69}
    	    \put(-5,142){0.68}
        	\put(-5,102.5){0.67}
        	\put(-5,62.5){0.66}
        	\put(-5,22.5){0.65}
        	\put(12.5,3){0}
        	\put(87,3){25}
        	\put(164,3){50}
        	\put(241,3){75}
        	\put(316,3){100}
            \put(328,13){$n$}
            \put(290,200){$\varepsilon = \ln 2$}
            \put(270,160){$k = 1$}
            \put(140,168){$k = 2$}
        \end{overpic}
        \caption{Efficiency of an $m=100$ bit classic Bloom filter storing $n$ items encoded using $k=1,2,\dots,20$ hashes. Maximal efficiency $\varepsilon_C^*(100)=0.96$ occurs at $n=1$ and $k=50$.}
        \label{fig:classic-efficiency}
    \end{figure}


\begin{corollary}
    \label{cor:asymp-eff}
    Asymptotic efficiencies for classic and standard Bloom filters are
        \begin{align*}
            \lim\limits_{m \rightarrow \infty} \varepsilon^*_S(m) = \ln 2 \approx 0.6931
            \quad   \mbox{and}  \quad
            \lim\limits_{m \rightarrow \infty} \varepsilon^*_C(m) = 1 \, ,
        \end{align*}
    respectfully.
\end{corollary}
\begin{proof}
    Both $\varepsilon^*_S(m)$ and $\varepsilon^*_C(m)$ are strictly increasing in $m$.
    Using an approximation for central binomial coefficients based on the Wallis product, $\binom{2m}{m} \sim \frac{4^{m}}{\sqrt{\pi m}}$ and
        \begin{equation*}
        \lim\limits_{m \rightarrow \infty} \varepsilon_C^* (m) = \lim\limits_{m\rightarrow\infty} \, \frac{1}{m} \log_2 \left( \frac{2^m}{\sqrt{\pi m / 2}} \right) = 1 \, .
        \end{equation*}
    For the standard Bloom filter,
        \begin{equation*}
        \lim\limits_{m \rightarrow \infty} \varepsilon_S^* (m) 
            = \lim\limits_{m \rightarrow \infty} \frac{1}{ \log_2 \left ( 1 + \frac{1}{m-1} \right)^m } = \ln 2 \,
        \end{equation*}
    which agrees with \cite{Weaver2014SatFilter}.
\end{proof}


\subsection{Valley Efficiency}
\label{subsec:valley-efficiency}

For fixed $m$ and $n$, peak efficiency is lowest when $n \approx m$ or the efficiency with $k$ and $k+1$ hash bits are equal; see Fig~\ref{fig:standard-efficiency} and Fig~\ref{fig:classic-efficiency}.
The following theorem uses approximation \eqref{eqn:fp_optimal_approx} to help characterize such intersection points.


\begin{theorem}
    \label{lem:hashing-efficency}
    For $k>0$,
    	\begin{equation}
        \label{eq:hashes-equal}
        \left(1 - e^{-k x} \right)^k = \left( 1 - e^{-(k+1)x} \right)^{k+1}
    	\end{equation}
    at $x=0$ and
        \begin{equation}
        x = \ln \left( \sqrt[k+1]{1+\sqrt[k+1]{1+\sqrt[k+1]{1+\cdots}}} \right) \, .
        \end{equation}
\end{theorem}
\begin{proof}
    For $x>0$, let $z=e^x$. 
    We claim solutions of the Lambert trinomial $z^{k+1} - z - 1 = 0$ satisfy equation~\eqref{eq:hashes-equal}, which is verified by the fact that
        \begin{equation*}
        \left(1 - \frac{1}{z^k} \right)^k = \left( \frac{z^k - 1}{z^k} \right)^k = \left( \frac{1/z}{z^k} \right)^{k} = z^{-k(k+1)}
        \end{equation*}
    agrees with
        \begin{equation*}
        z^{-k(k+1)} = \left( \frac{z}{z^{k+1}} \right)^{k+1} = \left( \frac{(z^{k+1} - z - 1 ) + z}{z^{k+1}} \right)^{k+1} = \left(1 - \frac{1}{z^{k+1}}\right)^{k+1} \, .
        \end{equation*}
    The formula $z = \sqrt[k+1]{1+z}$ can be used recursively to reach a numerical solution for $z$ and $x=\ln(z)$.
\end{proof}


\subsection{Compressed Filter Efficiency}
\label{sec:compressed-bloom-efficiency}

Filters with a parity bias have lower entropy and are good candidates for additional compression.
From Corollary~\ref{con:classic-max-efficiency}, a classic Bloom filter reaches maximum efficiency when a single item is encoded using half the filter bits, i.e., $n=1$ and $k=\frac{m}{2}$.
For a fixed $k < \frac{m}{2}$, there are $\binom{m}{k}$ ways a single item, $n=1$, can be stored and each possibility can be identified by an integer between 1 and $\binom{k}{k}$.
It takes at most $\log_2 \binom{m}{k}$ bits to represent these integers, and this achieves optimal lossless compression for the filter.
Since no information is lost, the compressed filter retains its false positive rate of $p=\binom{m}{k}^{-1}$ and achieves Walker's \cite{Walker2007Filters} information theoretic limit of efficiency $\varepsilon=1$ for all $k \le \frac{m}{2}$ when $n=1$.

Ultimately, the compression method determines efficiency and performance of a compressed Bloom filter, but the trade-offs may be worthwhile when transmitting or archiving a filter for long term storage.
For more details, see \cite{Mitzenmacher2002Compressed} where arithmetic codes are used to compress standard Bloom filters.

%% file: 05_discussion.tex

\section{Discussion}
\label{sec:discuss}

The multivariate union and intersection committee distributions can be used to construct inclusion-exclusion models for complex sampling problems.
If a sample is drawn from a committee distribution, the statistics in Section~\ref{subsec:estimators} can be used to estimate its hidden parameters.
Linear regression on the ratio of sampled factorial or binomial moments, as described in \cite[\S 4.4]{Burns2014Recursive}, can also be used for distribution fitting and parameter estimation for unknown occupancy models since the committee distribution is a GHFMD.

The models in this paper have assumed that urn occupancy is an exchangeable event, but that may not always be the case.
Kalinka~\cite{Kalinka2014Probability} shows how complicated the committee intersection model can become with the condition that one urn is twice as likely to become occupied as all other urns.
Rather than distributing a preset number of balls into the urns, Charalambides~\cite[Ch.4]{Charalambides2005Combinatorial} considers a model where the number of balls cast into an urn is an i.i.d. random variable and obtains analogous p.m.f.s and binomial moments for the resulting occupancy distributions.

Bloom filters are rarely described using urn models and, thus far, have yet to be characterized as committee problems.
The false-positive rate and efficiency formulas have clear presentations in terms of the raw and binomial moments of the committee distribution, and this probabilistic approach seems promising even though there are outstanding optimization problems.

What is clear, the exact false positive rates for both the classic and standard Bloom filters demonstrate that the recommended number $k=\frac{m}{n} \ln 2$ overshoots the optimal number of hash bits causing more work for a sub-optimal result.
A filter of $m=1024$ bits that stores $n=5$ items minimizes its false positive rate using $k^*_S=133$ hash bits for a standard Bloom filter and $k^*_C=124$ hash bits for a classic Bloom filter rather than the estimated $k^*=142$.
For the standard Bloom filter, this misconfiguration increases the false positive rate by 15.7\% and decreases the efficiency by 0.2\%.
Similarly for the classic Bloom filter, misconfiguration increases the false positive rate by 106.9\% and decreases the efficiency by 0.7\%.
For these values, the false-positive rates for this filter are on the order of $p=10^{-43}$ which is substantially smaller than the rates used in practical applications. 
Fig.~\ref{fig:fp-vs-k} shows that the optimal number of hash bits converges to the estimate as $n$ increases, so the main drawback from over hashing is the additional ``everytime work'' spent on each filter operation.

It is surprising how a subtle difference in hash construction can cause an asymptotic change in efficiency.
Each bit collision translates to information loss and a decrease in filter efficiency. 
This intuitively explains why Bloom-g filters \cite{Qiao2011One,Qiao2014Fast} -- which separate the marked bits for small groups of items -- are more efficient that standard Bloom filters.

From an information theoretic standpoint, a code that maximizes entropy will uniformly distribution its alphabet; for binary strings this means equal bit parity.
However, the efficiency lost by overfilling a Bloom filter -- i.e., having more 1s than 0s -- is greater than underfilling the same filter, so the expected bit sum of filter constructed using optimal parameters is slightly less than half the filter length.
This also explains why the classic Bloom filter is optimized using fewer hash bits than the standard Bloom filter; the classic Bloom filter is better at avoiding collisions and reaches bit parity quicker than the standard Bloom filter.
In practice, Bloom filters are usually build iteratively and extra items can be stored in a new filter once bit parity is reached to avoid overfilling.

Most practical questions concerning Bloom filters can be answered using Table~\ref{tab:bloom-as-occupancy}, the exact false-positive formulas \eqref{eqn:standard-bloom-fp} and \eqref{eqn:classic-bloom-fp}, the recursive relations \eqref{eqn:standard-bloom-recurrence} and \eqref{eqn:classic-bloom-recurrence}, the bounds in Corollary~\ref{cor:fp-upper-lower-bounds}, or the estimate \eqref{eqn:standard-fp-lower-bound}. 
Conjectures~\ref{con:optimal-ks} and \ref{con:classic-max-efficiency} highlight that certain fundamental questions about Bloom filter optimization remain open.

%% file: 06_acknowledgements.tex

\section*{Acknowledgements}
\label{sec:ack}

Special thanks to Colin McRae for a useful discussion on compressed Bloom filters and Ryan Speers for his feedback.